\documentclass[11pt,reqno]{amsart}
\usepackage[a4paper,margin=1.2in]{geometry}
\usepackage[T1]{fontenc}
\usepackage{lmodern}
\usepackage{amsmath,amssymb,amsthm,mathtools,mathrsfs,bm}
\usepackage{microtype}
\usepackage{xcolor}
\usepackage{enumitem}
\usepackage{booktabs,array,longtable,float}
\usepackage[colorlinks=true]{hyperref}
\usepackage[nameinlink,capitalize,noabbrev]{cleveref}
\usepackage[numbers,sort&compress]{natbib}

\definecolor{revisionred}{RGB}{190,0,0}

\hypersetup{
	colorlinks=true,
	linkcolor=blue,
	citecolor=blue,
	urlcolor=blue,
	pdftitle={Sharp decay characterization for the Oldroyd-B system with only velocity dissipation},
	pdfauthor={Zhi Chen, Mingwen Fei, Lvqiao Liu, Jiahong Wu},
	pdfkeywords={Oldroyd-B, velocity dissipation, sharp decay, critical Besov spaces}
}

\numberwithin{equation}{section}

\newtheorem{theorem}{Theorem}[section]
\newtheorem{proposition}[theorem]{Proposition}
\newtheorem{lemma}[theorem]{Lemma}
\newtheorem{corollary}[theorem]{Corollary}
\newtheorem{definition}[theorem]{Definition}

\newtheorem{remark}[theorem]{Remark}
\newtheorem{example}[theorem]{Example}

\newcommand{\R}{\mathbb R}
\newcommand{\N}{\mathbb N}
\newcommand{\Z}{\mathbb Z}
\newcommand{\Pj}{\mathbb P}
\newcommand{\Bdot}{\dot B}
\newcommand{\CL}{\widetilde L}
\newcommand{\dyad}{\dot\Delta}
\newcommand{\lowcut}{\dot S}
\newcommand{\Lamb}{\Lambda}
\newcommand{\diver}{\operatorname{div}}

\newcommand{\supp}{\operatorname{supp}}
\newcommand{\Id}{\operatorname{Id}}
\newcommand{\ee}{\mathrm e}
\newcommand{\dd}{\,\mathrm d}
\newcommand{\ii}{\mathrm i}
\newcommand{\low}{\ell}
\newcommand{\high}{h}
\newcommand{\cF}{\mathcal F}
\newcommand{\cS}{\mathcal S}
\newcommand{\cK}{\mathcal K}
\newcommand{\cL}{\mathcal L}

\newcommand{\cN}{\mathcal N}
\newcommand{\cE}{\mathcal E}
\newcommand{\cD}{\mathcal D}
\newcommand{\cX}{\mathcal X}
\newcommand{\cH}{\mathcal H}

\newcommand{\cR}{\mathcal R}
\newcommand{\cT}{\mathcal T}

\newcommand{\cA}{\mathcal A}
\newcommand{\frakB}{\mathfrak B}
\newcommand{\Besov}[3]{\dot B^{#1}_{#2,#3}}
\newcommand{\scal}{s_c}

\newcommand{\ip}[2]{\left\langle #1,#2\right\rangle}
\newcommand{\norm}[1]{\left\lVert #1\right\rVert}
\newcommand{\abs}[1]{\left|#1\right|}
\newcommand{\braket}[1]{\left\langle #1\right\rangle}

\title[Sharp decay characterization for Oldroyd-B with velocity dissipation]{Sharp decay characterization for the Oldroyd-B system with only velocity dissipation}
\author{Zhi Chen}
	\address{School of  Mathematics and Statistics, Anhui Normal University, Wuhu 241002, China}
	\email{zhichenmath@ahnu.edu.cn}
	\author{Mingwen Fei}
	\address{School of  Mathematics and Statistics, Anhui Normal University, Wuhu 241002, China}
	\email{mwfei@ahnu.edu.cn}
	\author{Lvqiao Liu}
	\address{School of  Mathematics and Statistics, Anhui Normal University, Wuhu 241002, China}
	\email{lvqiaoliu@ahnu.edu.cn}
	\author{Jiahong Wu}
	\address{Department of Mathematics, University of Notre Dame, Notre Dame, IN 46556, USA}
	\email{jwu29@nd.edu}
\date{\today}

\begin{document}

\begin{abstract}
We study the incompressible Oldroyd-B system on $\mathbb R^d$, $d\ge2$, in the regime in which the velocity is the only dissipated quantity: the stress tensor undergoes neither relaxation nor diffusion.
In this regime the operator $\mathbb P\operatorname{div}$ acting on symmetric stresses has an infinite-dimensional kernel consisting of exact stationary states, so no decay estimate can hold for the full stress tensor and the existing sharp-decay theory for dissipative systems does not apply.
We identify the active pair $U=(u,\Lambda^{-1}\mathbb P\operatorname{div}\tau)$ as the correct decaying observable and establish, in a critical hybrid Besov framework, a complete characterization of its large-time behavior.
For $-d/2<\sigma_1<\sigma\le d/2-1$ we prove that $\|U(t)\|_{\mathcal H^{\sigma}}=O(t^{-(\sigma-\sigma_1)/2})$ holds if and only if the low-frequency part of $U_0$ belongs to $\dot B^{\sigma_1}_{2,\infty}$, and, for $\sigma<d/2-1$, that the two-sided estimate $\|U(t)\|_{\mathcal H^{\sigma}}\simeq t^{-(\sigma-\sigma_1)/2}$ holds if and only if the low-frequency part of $U_0$ belongs to a nondegenerate subclass $\mathfrak B^{\sigma_1}_{2,\infty}$ of $\dot B^{\sigma_1}_{2,\infty}$.
No decay is assumed, or obtained, for the stress tensor itself; it enters the analysis only through its small critical norm, as a coefficient.
Two ingredients are new.
First, a frequency-resolved inverse estimate recovers the negative Besov regularity of the data from the decay of the solution; it is closed by a weighted discrete Gronwall lemma and controls the non-decaying passive stress through a localized commutator estimate for $[\Lambda^{-1}\mathbb P\operatorname{div},u\cdot\nabla]$.
Second, for $\sigma<d/2-1$, we show that the low-frequency Duhamel source has a vanishing dyadic tail, which upgrades the nonlinear error $U-U_L$ to $o(t^{-(\sigma-\sigma_1)/2})$ with no additional hypothesis: the Wiegner-type remainder property is obtained as a conclusion rather than imposed as an assumption.
\end{abstract}

\maketitle

\tableofcontents

\section{Introduction and main results}\label{sec:intro}

\subsection{The model}
The Oldroyd-B (OB) model, introduced by Oldroyd \cite{Oldroyd1958}, is among the most classical constitutive models for viscoelastic fluids. It describes the competition between viscous dissipation of the velocity and elastic response of the polymeric stress that characterizes non-Newtonian behavior. We consider the following family of incompressible Oldroyd-B systems on $[0,\infty)\times\R^d$, $d\ge2$:
\begin{equation}\label{eq:parametric-GOB}
	\begin{cases}
		\partial_tu+(u\cdot\nabla)u+\nu\Lamb^{2\kappa}u+\nabla P
		=\nu_1\diver\tau,\\
		\partial_t\tau+(u\cdot\nabla)\tau+a\tau
		+\eta\Lamb^{2\zeta}\tau+Q(\tau,\nabla u)
		=\nu_2D(u),\\
		\diver u=0,\\
		u(0,x)=u_0(x),\qquad \tau(0,x)=\tau_0(x).
	\end{cases}
\end{equation}
Here $\Lamb=(-\Delta)^{1/2}$, $u(t,x)=(u^1,\dots,u^d)(t,x)$ is the velocity field, $P$ is the scalar pressure, and $\tau(t,x)=(\tau^{ij}(t,x))_{d\times d}$ is the symmetric non-Newtonian part of the stress tensor, with $[\diver\tau]^i=\partial_j\tau^{ij}$\footnote{Throughout, $i,j=1,\dots,d$ and we use the Einstein summation convention over repeated indices.}.
The tensor $D(u)=\frac12(\nabla u+(\nabla u)^\top)$ is the symmetric part of the velocity gradient, where $(\nabla u)_{ij}=\partial_ju^i$, and $W(u)=\frac12(\nabla u-(\nabla u)^\top)$ is its skew-symmetric part. The bilinear form
\begin{equation*}
	Q(\tau,\nabla u)=\tau W(u)-W(u)\tau+b\big(D(u)\tau+\tau D(u)\big),\qquad b\in[-1,1],
\end{equation*}
is the standard objective derivative term; $b=0$ is the corotational case.
The constants $\nu_1,\nu_2>0$ are coupling parameters, $\nu\ge0$ and $\kappa\in(0,1]$ describe the velocity dissipation, $a\ge0$ is the stress-relaxation (damping) coefficient, and $\eta\ge0$, $\zeta\in(0,1]$ describe the stress diffusion.
We denote by $\Pj=\Id+\nabla(-\Delta)^{-1}\diver$ the Leray projection onto divergence-free fields.

The mathematical theory of \eqref{eq:parametric-GOB} depends decisively on which dissipative mechanisms are present; representative results are collected in Table~1.
\begin{table}[H]
	\centering
	\caption{Representative dissipative regimes for the incompressible OB system.}
	\begingroup
	\small
	\setlength{\tabcolsep}{3.5pt}
	\renewcommand{\arraystretch}{1.20}
	\begin{tabular}{@{}
			>{\raggedright\arraybackslash}p{0.275\textwidth}
			>{\raggedright\arraybackslash}p{0.275\textwidth}
			>{\raggedright\arraybackslash}p{0.485\textwidth}@{}}
		\toprule
		Parameter regime & Dissipative mechanism & Representative results \\
		\midrule
		$\nu>0$, $\kappa=1$,\newline
		$a>0$, $\eta=0$
		&
		velocity dissipation and stress relaxation
		&
		Strong solutions and critical-space theories
		\cite{GuillopeSaut1990,CheminMasmoudi2001,ChenMiao2008,ZhangFang2012,
			ZiFangZhang2014,FangZi2016}; corotational weak solutions
		\cite{LionsMasmoudi2000}; two-sided decay estimates
		\cite{HieberNguyenNichePerusato2026}.\\[3mm]
		
		$\nu>0$, $\frac12\le\kappa\le1$,\newline
		$a=\eta=0$
		&
		velocity dissipation only
		&
		Global small solutions for $\kappa=1$ \cite{Zhu2018,ChenHao2019,Zhai2021,LinWeiWu2022,ChenZhu2023};
		fractional velocity dissipation \cite{ZhaoWu2023};
		optimal upper decay rates \cite{Wan2019,ZhaiDecay2019}.\\[3mm]
		
		$\nu>0$, $\kappa=1$,\newline
		$a\ge0$, $\eta>0$, $\zeta=1$
\newline
		&
		velocity and stress diffusion
		&
		Two-dimensional global regularity with diffusive stress
		\cite{ConstantinKliegl2012}.
		\\[3mm]
		
		$\nu=0$,\newline
		$a>0$, $\eta>0$, $\zeta=1$
		&
		stress relaxation and diffusion, no velocity dissipation
		&
		Small-data global regularity in two and three dimensions
		\cite{ElgindiRousset2015,ElgindiLiu2015}.\\[3mm]
		
		$\nu=0$,\newline
		$a=0$, $\eta>0$, $\frac12\le\zeta\le1$\newline
		&
		fractional stress diffusion only
		&
		Sobolev stability for $\zeta\ge\frac12$ \cite{ConstantinWuZhaoZhu2021};
		critical hybrid Besov theory \cite{Chen,WuZhao2022};
		sharp decay \cite{WangWuXuZhong2022,ChenFeiLiuWu2026}.\\[3mm]
		
		anisotropic replacements of
		$\Lamb^{2\kappa}$ and $\Lamb^{2\zeta}$
		&
		mixed partial dissipation
		&
		Two- and three-dimensional results
		\cite{FengWangWu2022,LinWuBoardman2025}.\\
		\bottomrule
	\end{tabular}
	\endgroup
\end{table}

The coupling terms $\diver\tau$ and $D(u)$ are the source of the hidden dissipation exploited in the partially dissipative theory. The two stress mechanisms act very differently: relaxation ($a>0$) damps every component of $\tau$ but supplies no spatial derivative, whereas diffusion ($\eta>0$) regularizes $\tau$. When $a=\eta=0$, the velocity dissipation is transmitted, through the coupling, only to the projected stress component $\Pj\diver\tau$; it has no effect at all on $\ker(\Pj\diver)$.

Throughout this paper we work in the regime
\begin{equation}\label{eq:present-parameter-regime}
	\nu>0,\qquad \kappa=1,\qquad a=\eta=0,\qquad
	\nu_1>0,\qquad \nu_2>0,
\end{equation}
in which \eqref{eq:parametric-GOB} reduces to
\begin{equation}\label{eq:GOB}
	\begin{cases}
		\partial_tu+(u\cdot\nabla)u-\nu\Delta u+\nabla P=\nu_1\diver\tau,\\
		\partial_t\tau+(u\cdot\nabla)\tau+Q(\tau,\nabla u)=\nu_2D(u),\\
		\diver u=0,\\
		u(0,x)=u_0(x),\qquad \tau(0,x)=\tau_0(x).
	\end{cases}
\end{equation}
For \eqref{eq:GOB}, Zhu \cite{Zhu2018} established global smooth solutions for small data in $\R^3$, Chen and Hao \cite{ChenHao2019} developed the critical hybrid Besov theory, and Zhai \cite{Zhai2021} introduced related critical frameworks accommodating highly oscillatory velocity data. In two dimensions, the theory comprises the periodic result of Lin, Wei and Wu \cite{LinWeiWu2022} and the whole-space result of Chen and Zhu \cite{ChenZhu2023}. Upper decay rates were obtained by Wan \cite{Wan2019} and Zhai \cite{ZhaiDecay2019}; related three-dimensional decay results include \cite{HieberWenZi2019,ChenLiYao2026}. All of these decay results are one-sided: they provide upper bounds under a sufficient condition on the data.

\subsection{Sharp decay characterization and the stationary obstruction}
A \emph{sharp} decay theory asks for a complete answer to the following inverse question: \emph{which initial data produce exactly a given algebraic decay rate?}
For dissipative systems with a heat-like low-frequency behavior, Brandolese \cite{Brandolese2016} showed that the rate $t^{-(\sigma-\sigma_1)/2}$ is attained precisely by data whose low frequencies lie in the Besov space $\dot B^{\sigma_1}_{2,\infty}$, and that a two-sided rate corresponds to a nondegeneracy condition on the dyadic blocks of the data. This viewpoint has recently been extended to the compressible Navier-Stokes equations by Brandolese, Shou, Xu and Zhang \cite{BrandoleseShouXuZhang2024}, to partially dissipative hyperbolic systems by Shou, Xu and Zhang \cite{ShouXuZhang2025}, to viscous compressible fluids with capillarity by Shi, Shou and Xu \cite{ShiShouXu2026}, and, in our previous work \cite{ChenFeiLiuWu2026}, to the Oldroyd-B system with stress diffusion. In all of these settings every component of the solution decays, and the characterization concerns the whole state.

The velocity-dissipative system \eqref{eq:GOB} does not fit this paradigm. Let
\begin{equation}\label{eq:kernel-intro}
	\cK:=\ker(\Pj\diver)=\bigl\{K \text{ symmetric}:\ \Pj\diver K=0\bigr\}.
\end{equation}
For every sufficiently regular $K\in\cK$ the pair $(u,\tau)\equiv(0,K)$ is an exact stationary solution of \eqref{eq:GOB} (Proposition~\ref{thm:stationary-1}), and $\cK$ is infinite-dimensional: it contains all Hessians $\nabla^2\phi$ with $\phi\in\cS(\R^d)$ (Example~\ref{ex:Hessian-kernel}). Consequently \emph{no} estimate of the form $\|\tau(t)\|_X\to0$ can hold for small data in any reasonable space $X$, and a sharp decay statement for the full pair $(u,\tau)$ is false in general. This is a genuine structural obstruction, not a technical one; to the best of our knowledge, no sharp (necessary and sufficient) decay characterization is available for a fluid system with an infinite-dimensional family of stationary states.

The way out is to identify the part of the stress that is actually driven by the dissipation. Setting
\begin{equation}\label{active-z}
	\scal:=\frac d2-1,
	\qquad
	\cT:=\Lambda^{-1}\Pj\diver,
	\qquad
	z:=\cT\tau,
	\qquad
	U:=(u,z),
\end{equation}
we show in Section~\ref{sec:linear-chara} that every symmetric stress splits canonically as $\tau=\cR z+K$ with $K\in\cK$, where $\cR$ is an explicit order-zero right inverse of $\cT$ (Lemmas~\ref{lem:right-inverse} and~\ref{prop:quotient-decomp}). The \emph{active} component $z$ is a divergence-free vector field carrying exactly the same information as the equivalence class of $\tau$ modulo $\cK$, while the \emph{passive} component $K$ is invisible to the linearized dynamics. The active pair $U=(u,z)$ satisfies a closed linear system whose low-frequency part behaves like a heat semigroup (Proposition~\ref{thm:low-frequency-Fourier}) and whose high-frequency part is exponentially damped (Proposition~\ref{prop:high-frequency}). This makes $U$ the natural observable for a sharp decay theory. Since $u$ and $z$ have different regularities at high frequencies, the correct norms are hybrid: low frequencies are measured at a common index, high frequencies at the critical indices $d/2-1$ for $u$ and $d/2$ for $z$.

\subsection{Main results}
We begin by recalling the global existence theory in the critical hybrid Besov framework. The precise definitions of the homogeneous Besov spaces $\Bdot^s_{2,r}$, of the Chemin-Lerner spaces $\CL^q_T(\Bdot^s_{2,r})$, of the low-frequency and high-frequency parts $f^\low$, $f^\high$ relative to a fixed threshold $j_0<0$, and of the corresponding truncated seminorms are collected in Appendix~\ref{app:para}.

\begin{theorem}[\cite{ChenHao2019}]\label{thm:known-global}
	Let $d\ge2$. There exists $\varepsilon_0>0$ such that if $\diver u_0=0$, $\tau_0$ is symmetric, and
	\begin{equation}\label{eq:X0}
		\cX_0:=
		\norm{(u_0,\tau_0)^\low}_{\Bdot^{d/2-1}_{2,1}}
		+\norm{u_0^\high}_{\Bdot^{d/2-1}_{2,1}}
		+\norm{\tau_0^\high}_{\Bdot^{d/2}_{2,1}}
		\le\varepsilon_0,
	\end{equation}
	then \eqref{eq:GOB} admits a unique global solution $(u,\tau)$ with $\cX(t)<\infty$ for all $t>0$, where
	\begin{align}
		\cX(t):={}&
		\norm{(u,\tau)^\low}_{\CL^\infty_t(\Bdot^{d/2-1}_{2,1})}
		+\norm{u^\high}_{\CL^\infty_t(\Bdot^{d/2-1}_{2,1})}
		+\norm{\tau^\high}_{\CL^\infty_t(\Bdot^{d/2}_{2,1})}\notag\\
		&+\norm{(u,z)^\low}_{L^1_t(\Bdot^{d/2+1}_{2,1})}
		+\norm{u^\high}_{L^1_t(\Bdot^{d/2+1}_{2,1})}
		+\norm{z^\high}_{L^1_t(\Bdot^{d/2}_{2,1})}
		\le C\cX_0,\label{6364553}
	\end{align}
	with $z=\cT\tau$ as in \eqref{active-z} and $C>0$ independent of $t$.
\end{theorem}

\begin{remark}
	The existence theorem and the critical estimates are taken from \cite{ChenHao2019}; related $L^p$ theories and decay estimates are developed in \cite{Zhai2021,ZhaiDecay2019}. The hybrid stress norm in \eqref{eq:X0} is equivalent to the norm of $\dot B^{d/2-1}_{2,1}\cap\dot B^{d/2}_{2,1}$. The dissipative bounds on the active variable $z$ in \eqref{6364553} are also re-derived, in a form adapted to our purposes, in Proposition~\ref{prop:act-Lyap} below. In what follows, $\varepsilon_0$ is allowed to be decreased finitely many times; all smallness conditions are of the form $\cX_0\le\varepsilon_0$ with $\varepsilon_0$ depending only on $d,\nu,\nu_1,\nu_2,b,\sigma_1$ and on the Littlewood-Paley threshold.
\end{remark}

We now introduce the hybrid norms in which the decay is measured.

\begin{definition}\label{def:hybrid-a-1}
	Fix a homogeneous Littlewood-Paley decomposition with threshold $j_0<0$ as in Appendix~\ref{app:para}. For $\rho\le\scal$ and $U=(u,z)$, set
	\begin{align}
		\norm{U}_{\cH^\rho}
		&:={}
		\norm{U^{\low}}_{\Besov{\rho}{2}{1}}
		+\norm{u^{\high}}_{\Besov{\scal}{2}{1}}
		+\norm{z^{\high}}_{\Besov{\scal+1}{2}{1}}.
		\label{eq:Hrho}
	\end{align}
	Thus $\cH^\rho$ is the $p=2$ hybrid active norm with low-frequency target index $\rho$; its high-frequency part is always the critical one. Note that $\norm{U}_{\cH^{\scal}}$ is exactly the active part of the critical norm in \eqref{6364553}.
\end{definition}

\begin{definition}\label{def:frakB}
	Let $s\in\R$. A tempered distribution $F$ belongs to $\frakB^{s}_{2,\infty}$ if $F\in\Bdot^{s}_{2,\infty}$ and there exist constants $c_0>0$, $M\in\N$, and a strictly decreasing sequence of integers $(j_k)_{k\ge1}$ with $j_k\to-\infty$ such that
	\begin{equation}\label{eq:frakB-def}
		1\le j_k-j_{k+1}\le M,
		\qquad
		2^{sj_k}\norm{\dyad_{j_k}F}_{L^2}\ge c_0
		\qquad\text{for all }k\ge1.
	\end{equation}
\end{definition}

The class $\frakB^{s}_{2,\infty}$ consists of those elements of $\Bdot^{s}_{2,\infty}$ whose dyadic blocks do not degenerate along a sequence of low frequencies with bounded gaps; it is the natural nondegeneracy condition for two-sided decay, cf.\ \cite{Brandolese2016,BrandoleseShouXuZhang2024}.

Our main result is the following complete characterization of the decay of the active pair.

\begin{theorem}\label{thm:non-trans}
	Let $d\ge 2$ and let $(u,\tau)$ be the global solution of \cref{thm:known-global}, with $\varepsilon_0$ in \eqref{eq:X0} sufficiently small. Set $U=(u,\cT\tau)$ and let $U_L(t)=e^{t\cL}U_0$, where $\cL$ is the active linear operator defined in \eqref{64238682}. Assume
	\begin{equation}
		-\frac d2<\sigma_1<\sigma\le\scal
		\label{eq:main-assumptions}
	\end{equation}
	and set
	\begin{equation}
		\alpha:=\frac{\sigma-\sigma_1}{2}>0.
		\label{eq:alpha}
	\end{equation}
	\begin{enumerate}[label=\textup{(\roman*)},leftmargin=2.2em]
		\item \textbf{Upper characterization.}
		The following two statements are equivalent:
		\begin{align}
			&U_0^{\low}\in\Besov{\sigma_1}{2}{\infty};
			\label{eq:negative-datum}\\
			&\text{there exists } C>0 \text{ such that }\quad
			\norm{U(t)}_{\cH^\sigma}
			\le C\langle t\rangle^{-\alpha}
			\quad\text{for all } t\ge0.
			\label{eq:upper-nonlinear}
		\end{align}
		
		\item \textbf{Automatic Wiegner-type property.}
		If $\sigma<\scal$ and either of the equivalent conditions in part~\textup{(i)} holds, then
		\begin{equation}
			\norm{U(t)-U_L(t)}_{\cH^\sigma}
			=o\!\left(t^{-\alpha}\right)
			\qquad(t\to\infty).
			\label{eq:auto-wiegner}
		\end{equation}
		
		\item \textbf{Two-sided characterization.}
		If $\sigma<\scal$, the following two statements are equivalent:
		\begin{align}
			&U_0^{\low}\in\frakB^{\sigma_1}_{2,\infty};
			\label{eq:nondeg-datum}\\
			&\text{there exist } c,C>0 \text{ and } t_0\ge1 \text{ such that }\notag\\
			&\qquad c\,t^{-\alpha}
			\le \norm{U(t)}_{\cH^\sigma}
			\le C\,t^{-\alpha}
			\quad\text{for all } t\ge t_0.
			\label{eq:two-sided-nonlinear}
		\end{align}
	\end{enumerate}
\end{theorem}

Here $\langle t\rangle=(1+t^2)^{1/2}$. Several comments are in order.

\begin{remark}\label{rem:main}
	\begin{enumerate}[label=\textup{(\alph*)},leftmargin=2.2em,itemsep=2pt]
		\item No assumption whatsoever is made on $\tau_0$ beyond the smallness \eqref{eq:X0}; in particular, the passive component $\tau_0-\cR z_0\in\cK$ is arbitrary within that class. Correspondingly, the theorem asserts nothing about the decay of $\tau$, which in general does not decay (Proposition~\ref{thm:stationary-1}). The stress tensor enters the proof only through its small critical norm, where it acts as a coefficient.
		\item The low-frequency condition \eqref{eq:negative-datum} is placed on the active datum $U_0=(u_0,\cT\tau_0)$ only. This is strictly weaker than the condition $(u_0,\tau_0)^\low\in\Bdot^{\sigma_1}_{2,\infty}$ used in the upper-bound results \cite{Wan2019,ZhaiDecay2019}, and the theorem shows that it is the optimal one.
		\item Part~\textup{(ii)} shows that the nonlinear solution is asymptotically indistinguishable, at the rate $t^{-\alpha}$, from the linear active evolution $U_L$. In the classical approach to lower bounds (Wiegner's argument for the Navier-Stokes equations) a remainder estimate of this type is an additional hypothesis or a separate input; here it is a consequence of the critical $\ell^1$ structure of the solution.
		\item Part~\textup{(i)} covers the critical level $\sigma=\scal$, where the full critical norm of the active pair is characterized. Parts~\textup{(ii)} and~\textup{(iii)} are stated for $\sigma<\scal$; see Remark~\ref{rem:critical-level} for a discussion of the endpoint.
		\item Adding a stress-relaxation term $a\tau$, $a>0$, damps the passive stress and simplifies the nonlinear estimates; we expect an analogous characterization to hold after adapting the low-frequency conditions to the slow spectral mode. An extension to the mixed critical $L^2$-$L^p$ framework of \cite{Zhai2021,ChenFeiLiuWu2026} appears feasible once the corresponding mixed-norm product estimates are available. We also expect that $\tau(t)$ converges, as $t\to\infty$, to a stationary stress in $\cK$; this will be addressed elsewhere.
	\end{enumerate}
\end{remark}

\subsection{Novelty and strategy of the proof}
The present work addresses a dissipative structure different from that of our earlier paper \cite{ChenFeiLiuWu2026}, which treated the inviscid Oldroyd-B system with stress diffusion. There, as in \cite{Brandolese2016,BrandoleseShouXuZhang2024,ShouXuZhang2025,ShiShouXu2026}, every component of the solution decays. Here the stress contains an infinite-dimensional passive component that does not decay, and this changes the problem in three ways: the decaying observable has to be identified; the non-decaying stress cannot be discarded from the nonlinear analysis, because it enters the equation for $U$ through $[\cT,u\cdot\nabla]\tau$ and $\cT Q(\tau,\nabla u)$; and the inverse problem, recovering the regularity of the data from the decay of the solution, must be solved while controlling the passive stress only in its small critical norm. The main contributions are the following.

\begin{enumerate}[label=\textup{(\roman*)},leftmargin=2.2em,itemsep=3pt]
	\item \textbf{The stationary obstruction and the active quotient.}
	We identify the infinite-dimensional kernel $\cK=\ker(\Pj\diver)$ of exact stationary stresses and construct a canonical active/passive decomposition $\tau=\cR z+K$ (Section~\ref{sec:linear-chara}). This explains why a decay characterization for $(u,\tau)$ is false in general and singles out $z=\Lamb^{-1}\Pj\diver\tau$, equivalently the active pair $U=(u,z)$, as the correct object. The linear active system is then analyzed exactly: after a constant symmetrization, its low-frequency semigroup is a heat factor times a uniformly bounded oscillatory matrix with uniformly bounded inverse (Lemma~\ref{lem:low-semigroup}), which yields two-sided heat-type dyadic bounds and a complete linear characterization (Theorems~\ref{thm:li-upp-} and~\ref{thm:linear-t-s}).
	
	\item \textbf{Active nonlinear decay without decay of the passive stress.}
	We prove a critical Lyapunov inequality for the active pair alone (Proposition~\ref{prop:act-Lyap}), in which the full stress appears only through its critical amplitude. Combined with the propagation of negative active regularity (Theorem~\ref{thm:negative-p}) and an interpolation argument in the spirit of \cite{Xin-Xu-2021}, this gives the optimal upper decay of $U$ (Theorem~\ref{thm:nonlinear-upper}) without any decay assumption on $\tau$.
	
	\item \textbf{Recovering the initial regularity from the decay.}
	The direction ``decay $\Rightarrow$ data regularity'' is the heart of a sharp characterization, and it is where the passive stress is most dangerous: the critical estimates alone do not control the low-frequency Duhamel source in the negative norm $\Bdot^{\sigma_1}_{2,\infty}$. Lemma~\ref{lem:same-order} resolves this by a frequency-resolved analysis. For each dyadic index $j$ we introduce the dissipative amplitude $X_j$ of the solution and the source amplitude $Y_j$, and derive an inequality of the form
	\[
	X_j\le C\bigl(M+\cX_0^2\bigr)+C\varepsilon_j\sum_{k<j}2^{-(j-k)\gamma_1}X_k+C\sum_{k\ge j}2^{-(k-j)\gamma_2}\varepsilon_kX_k ,
	\]
	where $(\varepsilon_j)$ is an $\ell^1$ sequence controlled by the critical norm of $(u,\tau)$ and $\gamma_1,\gamma_2>0$. The interactions with the passive stress are handled by a localized commutator estimate for $[\cT,u\cdot\nabla]$ in which the first moment of the kernel of $\cT$ absorbs the derivative falling on the stress. The inequality is closed by a weighted discrete Gronwall lemma (Lemma~\ref{key-le-10}); evaluating the solution at the diffusion times $t_j=2^{-2j}$ then yields $U_0^\ell\in\Bdot^{\sigma_1}_{2,\infty}$, together with a same-order bound for the low-frequency difference between the nonlinear and linear solutions. The argument includes the critical endpoint $\sigma=\scal$.
	
	\item \textbf{A vanishing source tail and the automatic Wiegner-type property.}
	The critical $\ell^1$ summability allows the stress to be approximated by finitely many frequency blocks. Combining this with the support properties of Bony's decomposition, we prove in Lemma~\ref{lem:c0-source} that the low-frequency source has a vanishing dyadic tail,
	\[
	2^{j\sigma_1}
	\norm{\dyad_j\cN(u,\tau)}_{L^1(0,\infty;L^2)}
	\longrightarrow 0
	\qquad (j\to-\infty).
	\]
	A three-zone decomposition around the moving parabolic scale $2^j\simeq t^{-1/2}$ then converts this into $\norm{U(t)-U_L(t)}_{\cH^\sigma}=o(t^{-\alpha})$ for $\sigma<\scal$ (Lemma~\ref{lem:auto-wiegner}). Thus the Wiegner-type remainder property is a conclusion rather than a hypothesis, and the sharp two-sided linear characterization transfers to the nonlinear active solution despite the possible non-decay of the passive stress.
\end{enumerate}

\subsection{Notation and organization}
We write $A\lesssim B$ if $A\le CB$ for a constant $C$ depending only on $d,\nu,\nu_1,\nu_2,b$, on the Besov indices involved, and on the Littlewood-Paley threshold $j_0$; $A\simeq B$ means $A\lesssim B\lesssim A$. The letters $c,C$ denote positive constants of this type whose value may change from line to line. We write $\norm{\cdot}_2$ for the $L^2(\R^d)$ norm and $\ip\cdot\cdot$ for the $L^2$ inner product. For a time-dependent function $f$ and $1\le q\le\infty$, $L^q_t L^2_x$ denotes $L^q(0,\infty;L^2(\R^d))$ unless a finite horizon $T$ is indicated.

The paper is organized as follows. Section~\ref{sec:linear-chara} constructs the active/passive decomposition and establishes the sharp decay characterization for the linear active system: low- and high-frequency estimates, the necessary and sufficient condition for upper decay, and the two-sided decay. Section~\ref{sec:kernel} proves the critical Lyapunov inequality for the active pair. Section~\ref{sec:non-a} derives the nonlinear upper decay through the propagation of negative active regularity. Section~\ref{sec:nonlinear-transfer} contains the inverse estimate, the vanishing source tail, the automatic Wiegner-type property, and the proof of \cref{thm:non-trans}. Appendix~\ref{app:para} collects the paradifferential tools, and Appendix~\ref{app:critical-hybrid-products} the critical hybrid product estimates and the discrete Gronwall lemma.


\section{Sharp decay characterization for the linear
 system}\label{sec:linear-chara}

\subsection{The stationary stress kernel}
We first make precise the kernel introduced in \eqref{eq:kernel-intro}. Define
\begin{equation*}
	\cK:=\left\{K\in\cS'_h(\R^d;\R^{d\times d}_{\mathrm{sym}}):\Pj\diver K=0\right\},
\end{equation*}
where $\cS'_h$ is the space of homogeneous tempered distributions recalled in Appendix~\ref{app:para}. Since $\Pj f=0$ if and only if $f$ is a gradient, any $K\in\cK$ satisfies $\diver K=\nabla q$ for some scalar distribution $q$, modulo the usual polynomial ambiguity of homogeneous spaces.

The stress tensor itself does not decay in general. The following proposition shows that this failure is unavoidable; it forces us to split $\tau$ into an active (decaying) part and a passive (possibly non-decaying) part, and all decay statements in this paper concern the active part.
\begin{proposition} \label{thm:stationary-1}
	
	Let $K\in\cK$ be a symmetric tensor field for which the pair below is an admissible state of \eqref{eq:GOB} (for instance $K\in\Bdot^{d/2-1}_{2,1}\cap\Bdot^{d/2}_{2,1}$). Then
	\begin{equation}\label{eq:stationary-solution}
		u(t,x)\equiv0,
		\qquad
		\tau(t,x)\equiv K(x)
	\end{equation}
	is an exact solution of the OB system \eqref{eq:GOB}, with pressure $P=\nu_1q$ whenever $\diver K=\nabla q$. Consequently, if a Banach space $X$ contains a nonzero sufficiently regular element
	$K\in\cK$ whose scalar multiples are admissible initial stresses, no
	estimate of the form
	\[
	\norm{\tau(t)}_X\longrightarrow0\qquad(t\to\infty)
	\]
	can hold uniformly for all sufficiently small data in $X$.
\end{proposition}

\begin{proof}
	For $u\equiv0$, one has
	\[
	(u\cdot\nabla)u=0,
	\qquad
	D(u)=0,
	\qquad
	Q(\tau,\nabla u)=Q(K,0)=0.
	\]
	
	Hence the stress equation reduces to $\partial_t\tau=0$, which holds for $\tau\equiv K$. The momentum equation becomes
	\[
	\nabla P=\nu_1\diver K=\nu_1\nabla q,
	\]
	which is satisfied by $P=\nu_1q$. The divergence constraint is automatic. If a universal stress-decay estimate were valid, applying it to \eqref{eq:stationary-solution} would force $\norm K_X=0$ for any $K\in X\cap\cK$, which contradicts the existence of a nonzero such $K$.
\end{proof}
We next provide a concrete example demonstrating that 
$\cK$ is nonempty and, indeed, infinite-dimensional.
\begin{example}\label{ex:Hessian-kernel}
	For any nonzero $\phi\in\cS(\R^d)$,
	\[
	K=\nabla^2\phi
	\]
	is symmetric and satisfies
	\[
	\diver K=\nabla\Delta\phi,
	\qquad
	\Pj\diver K=0.
	\]
	
	Thus $\cK$ is infinite-dimensional and contains arbitrarily small smooth rapidly decaying tensors.
	
\end{example}

\subsection{A canonical active/passive decomposition}

 For a divergence-free vector field $z$, define the symmetric order-zero tensor operator
\begin{equation*}\label{eq:R-operator}
	\cR z:=-\Lamb^{-1}\bigl(\nabla z+(\nabla z)^\top\bigr).
\end{equation*}

\begin{lemma}\label{lem:right-inverse}
	If $\diver z=0$, then
	\begin{equation}\label{eq:right-inverse}
		\diver(\cR z)=\Lamb z.
	\end{equation}
	Moreover, for any $s\in\R$ and $1\le r\le\infty$,
	\begin{equation}\label{eq:R-bounded}
		\norm{\cR z}_{\Bdot^s_{2,r}}\le C\norm z_{\Bdot^s_{2,r}}.
	\end{equation}
\end{lemma}
\begin{proof}
	Since $\diver z=0$, one has
	\begin{align*}
		\bigl(\diver(\cR z)\bigr)_i
		&=-\Lamb^{-1}\sum_j\partial_j(\partial_i z_j+\partial_jz_i)\\
		&=-\Lamb^{-1}\bigl(\partial_i\diver z+\Delta z_i\bigr)
		=-\Lamb^{-1}\Delta z_i
		=\Lamb z_i.
	\end{align*}
	
	The bound \eqref{eq:R-bounded} follows because every component of $\cR$ is a finite linear combination of Riesz transforms.
\end{proof}
\begin{lemma}\label{prop:quotient-decomp}
	For every symmetric stress tensor $\tau$, recall from \eqref{active-z} that $z=\Lamb^{-1}\Pj\diver\tau$ and let
	\begin{equation*} 
		K=\tau-\cR z,
	\end{equation*}
	then $K\in\cK$. 
	The decomposition of $\tau$ is unique modulo $\cK$. In particular,
	\[
	\tau_1-\tau_2\in\cK
	\quad\Longleftrightarrow\quad
	\Lamb^{-1}\Pj\diver\tau_1=
	\Lamb^{-1}\Pj\diver\tau_2.
	\]
\end{lemma}

\begin{proof}
	By Lemma \ref{lem:right-inverse}, we have
	\[
	\Pj\diver K
	=\Pj\diver\tau-\Pj\diver(\cR z)
	=\Lamb z-\Lamb z=0.
	\]
	Hence $K\in\cK$. The remaining assertions follow immediately from the definition.
	
\end{proof}

\begin{remark}
	For any Besov index $(s,r)$, one may have
	\begin{align*}
		\|\cR z\|_{{\Bdot^s_{2,r}}}\approx\|z\|_{{\Bdot^s_{2,r}}}.
	\end{align*}
	Indeed, by
	\eqref{eq:R-bounded}, $\norm{\cR z}_{\Bdot^s_{2,r}}\le C\norm z_{\Bdot^s_{2,r}}$.
	Conversely, $\cT\cR z=z$ by \eqref{eq:right-inverse}, and the order-zero
	boundedness of $\cT$ gives $\norm z_{\Bdot^s_{2,r}}\le C\norm{\cR z}_{\Bdot^s_{2,r}}$.
\end{remark}

\subsection{Linear part of OB system}
For system \eqref{eq:GOB}, the linear part can be written as
\begin{equation}\label{eq:linear-1}
	\begin{cases}
		\partial_tu-\nu\Delta u+\nabla P=\nu_1\diver\tau,\\
		\partial_t\tau=\nu_2D(u),\\
		\diver u=0.
	\end{cases}
\end{equation}

Applying $\Pj$ to the first equation of \eqref{eq:linear-1} and using $\Pj\diver\tau=\Lamb z$ gives
\begin{equation*} 
	\partial_tu+\nu\Lamb^2u-\nu_1\Lamb z=0.
\end{equation*}
Since $\operatorname{div}D(u)=\frac12\Delta u$ for divergence-free $u$,
applying $\Lambda^{-1}\mathbb P\operatorname{div}$ to the stress equation yields
\begin{equation*} 
	\partial_tz+\frac{\nu_2}{2}\Lamb u=0.
\end{equation*}
Thus, the system \eqref{eq:linear-1} becomes
\begin{equation}\label{eq:l-1}
	\begin{cases}
	\partial_tu+\nu\Lamb^2u-\nu_1\Lamb z=0,\\
		\partial_tz+\frac{\nu_2}{2}\Lamb u=0,\\
		\diver u=0.
	\end{cases}
\end{equation}

Set
\begin{equation}\label{eq:cstar-Y}
	c_*:=\sqrt{\frac{\nu_1\nu_2}{2}},
	\qquad
	Y_1:=\sqrt{\frac{\nu_2}{2}}\,u,
	\qquad
	Y_2:=\sqrt{\nu_1}\,z,
	\qquad
	Y=(Y_1,Y_2),
\end{equation}
then
\begin{equation*}
	C^{-1}\abs U\le\abs Y\le C\abs U.
\end{equation*}
Taking the Fourier transform of $Y$ and writing $r=|\xi|$, one obtains
\begin{equation*} 
	\partial_t\widehat Y=A(r)\widehat Y,
	\qquad
	A(r)=
	\begin{pmatrix}
		-\nu r^2&c_*r\\
		-c_*r&0
	\end{pmatrix}.
\end{equation*}
In the range $0<r<2c_*/\nu$, the characteristic roots are
\begin{equation*} 
	\lambda_\pm(r)=-\frac{\nu r^2}{2}\pm\ii b(r), \quad\text{with} \quad b(r):=r\sqrt{c_*^2-\frac{\nu^2r^2}{4}}.
\end{equation*}


We now give the exact  semigroup for the  low-frequency of the linear system \eqref{eq:l-1} in the following Lemma. 
\begin{lemma}\label{lem:low-semigroup}
	
	For $0<r<2c_*/\nu$, one has
	\begin{equation}\label{eq:low-semigroup}
		e^{tA(r)}
		=\ee^{-\nu r^2t/2}
		\left[
		\cos(b(r)t)I+
		\frac{\sin(b(r)t)}{b(r)}B(r)
		\right],
	\end{equation}
where 
\begin{equation*} 
		B(r) 
		=
		\begin{pmatrix}
			-\nu r^2/2&c_*r\\
			-c_*r&\nu r^2/2
		\end{pmatrix}.
	\end{equation*}
\end{lemma}

\begin{proof}
	Noting  $A(r)=-\nu r^2I/2+B(r)$
	and $B(r)^2=-b(r)^2I$, the even and odd parts of the exponential series give
	\begin{align*}
		\ee^{tB}
		&=\sum_{k=0}^\infty\frac{t^{2k}B^{2k}}{(2k)!}
		+\sum_{k=0}^\infty\frac{t^{2k+1}B^{2k+1}}{(2k+1)!}\\
		&=\sum_{k=0}^\infty\frac{(-1)^k(bt)^{2k}}{(2k)!}I
		+\frac{B}{b}\sum_{k=0}^\infty\frac{(-1)^k(bt)^{2k+1}}{(2k+1)!}\\
		&=\cos(bt)I+\frac{\sin(bt)}bB.
	\end{align*}
	This proves \eqref{eq:low-semigroup}.
\end{proof}

\subsection{{Low-frequency estimates}}
The following proposition establishes uniform two-sided bounds for the low-frequency of linear system \eqref{eq:linear-1}.

\begin{proposition}\label{thm:low-frequency-Fourier}
	Let $U_L=(u_L,z_L)$ solve \eqref{eq:l-1}.
	Set
	\begin{equation*} 
		0<r_0 <\frac{2c_*}{\nu}.
	\end{equation*}
	There exists $C_0\ge1$, depending only on $\nu,\nu_1,\nu_2,r_0$, such that for all $t\ge0$ and $0<|\xi|\le r_0$,
	\begin{equation}\label{eq:two-Fourier}
		C_0^{-1}\ee^{-\nu|\xi|^2t/2}\abs{\widehat U_0(\xi)}
		\le\abs{\widehat U_L(t,\xi)}
		\le C_0\ee^{-\nu|\xi|^2t/2}\abs{\widehat U_0(\xi)}.
	\end{equation}
\end{proposition}

\begin{proof}
	By \eqref{eq:low-semigroup}, one has
		\[
		\widehat{Y}_L(t,\xi)
		=e^{-\nu r^2t/2}M(t,r)\widehat{Y}_0(\xi),
		\]
		with
		$$M(t,r)
		=\cos(b(r)t)I+\sin(b(r)t)\frac{B(r)}{b(r)}.$$
		Since $r_0<2c_*/\nu$,
		\[
		\frac{b(r)}{r}
		\ge \sqrt{c_*^2-\frac{\nu^2r_0^2}{4}}>0,
		\qquad 0<r\le r_0,
		\]
		and hence $B(r)/b(r)$ is uniformly bounded. Moreover,
		$B(r)^2=-b(r)^2I$ gives
		\[
		M(t,r)^{-1}
		=\cos(b(r)t)I-\sin(b(r)t)\frac{B(r)}{b(r)}.
		\]
        
		Thus both $M(t,r)$ and its inverse are uniformly bounded for
		$t\ge0$ and $0<r\le r_0$, which yields
	\[
	C^{-1}\abs{\widehat Y_0}
	\le\abs{M(t,r)\widehat Y_0}
	\le C\abs{\widehat Y_0}.
	\]
		The fixed invertible change of variables \eqref{eq:cstar-Y}
		then yields \eqref{eq:two-Fourier}.
	 
\end{proof}

Choose some negative $j_0$ such that
\begin{equation}\label{eq:j0-low}
	\frac83 2^{j_0}\le r_0.
\end{equation}
Then every Fourier annulus, supporting $\dyad_j$ with $j\le j_0$, is contained in the region of Proposition \ref{thm:low-frequency-Fourier}.

\begin{corollary}\label{cor:dyadic-two-sided}
	There exist constants $c,C>0$ such that for any $j\le j_0$ and $t\ge0$,
	\begin{equation}\label{eq:d-two-s}
		c\ee^{-C2^{2j}t}\norm{\dyad_jU_0}_{L^2}
		\le\norm{\dyad_jU_L(t)}_{L^2}
		\le C\ee^{-c2^{2j}t}\norm{\dyad_jU_0}_{L^2}.
	\end{equation}
	
	More generally, for any $m\ge0$,
	\begin{equation*}
		\norm{\Lamb^m\dyad_jU_L(t)}_{L^2}
		\le C2^{jm}\ee^{-c2^{2j}t}\norm{\dyad_jU_0}_{L^2}.
	\end{equation*}
\end{corollary}

\begin{proof}
	On the support of $\dyad_j$, one has $\frac34\,2^j\le|\xi|\le\frac83\,2^j$. Apply \eqref{eq:two-Fourier}, square, integrate, and use Plancherel's theorem. The derivative estimate follows by inserting the factor $|\xi|^m\simeq2^{jm}$.
\end{proof}

The nonlinear analysis uses the following consequence of the semigroup bound.
Let  $U=(u,z)$ solve
\begin{equation}\label{eq:forced-linear}
	\partial_tU=\cL U+ \mathcal{F},
	\qquad U(0)=U_0,
\end{equation}
where $\mathcal{F}=(F,G)^T$ is a vector function and  the active linear operator   $\cL$ is given by
\begin{align}\label{64238682}
	\cL:
	=\left(
	\begin{array}{cc}
		-\nu\Lambda^2 & \nu_1 \Lambda\\
		-\frac{\nu_2}{2}\Lambda & 0 \\
	\end{array}
	\right).
\end{align}
\begin{lemma}\label{lem:forced-low}
	Let $U$ solve \eqref{eq:forced-linear}.
	For any $s\in\R$, $\varrho\in\{1,\infty\}$ and $T>0$, it holds
	\begin{equation}\label{5625222}
		\norm{U^\low}_{\CL^\infty_T(\Bdot^s_{2,\varrho})}
		+\norm{U^\low}_{\CL^1_T(\Bdot^{s+2}_{2,\varrho})}
		\le C\left(
		\norm{U_0^\low}_{\Bdot^s_{2,\varrho}}
		+\norm{\mathcal{F}^\low}_{\CL^1_T(\Bdot^s_{2,\varrho})}
		\right),
	\end{equation}
	with $C$ independent of $T$.
\end{lemma}

\begin{proof}
	Duhamel's formula and the upper estimate in \eqref{eq:d-two-s} give, for $j\le j_0$,
	\begin{equation}\label{eq:block-Duhamel}
		\norm{\dyad_jU(t)}_2
		\le C\ee^{-c2^{2j}t}\norm{\dyad_jU_0}_2
		+C\int_0^t\ee^{-c2^{2j}(t-s)}\norm{\dyad_j\mathcal{F}(s)}_2\dd s.
	\end{equation}
	Taking the supremum in $t$ yields
	\[
	\norm{\dyad_jU}_{L^\infty_TL^2}
	\le C\norm{\dyad_jU_0}_2
	+C\norm{\dyad_j\mathcal{F}}_{L^1_TL^2}.
	\]
	Integrating \eqref{eq:block-Duhamel} in $t$, using interchange the order of integration and
	\[
	\int_0^T\ee^{-c2^{2j}t}\dd t\le C2^{-2j},
	\qquad
	\int_s^T\ee^{-c2^{2j}(t-s)}\dd t\le C2^{-2j},
	\]
	gives
	\[
	2^{2j}\norm{\dyad_jU}_{L^1_TL^2}
	\le C\norm{\dyad_jU_0}_2
	+C\norm{\dyad_j\mathcal{F}}_{L^1_TL^2}.
	\]
	Multiplying by $2^{js}$ and taking the $\ell^\varrho$ norm over $j\le j_0$ completes the proof.
\end{proof}

\subsection{High-frequency estimates}

Although the algebraic decay rate is selected by the low-frequency, the correct
full norm is hybrid.  At high-frequency the velocity and the active stress do not
have the same regularity: the natural pair is
$
(u,q)~\text{with}~q:=\Lambda z=\mathbb P\operatorname{div}\tau.
$
The next proposition is the
precise estimate needed in the linear characterization.

\begin{proposition}\label{prop:high-frequency}
		Let $U_L=(u_L,z_L)$ solve \eqref{eq:l-1}.
	There exist  constants $c,C>0$, depending only on
	$\nu,\nu_1,\nu_2$ and on the dyadic partition, such that
	\begin{align}
		&\sup_{t\ge0}\ee^{ct}
		\left(
		\norm{u_L^\high(t)}_{\Bdot^{d/2-1}_{2,1}}
		+\norm{z_L^\high(t)}_{\Bdot^{d/2}_{2,1}}
		\right)
		+
		\int_0^\infty\ee^{ct}
		\left(
		\norm{u_L^\high(t)}_{\Bdot^{d/2+1}_{2,1}}
		+\norm{z_L^\high(t)}_{\Bdot^{d/2}_{2,1}}
		\right)\dd t\nonumber\\
		&\le C
		\left(
		\norm{u_0^\high}_{\Bdot^{d/2-1}_{2,1}}
		+\norm{z_0^\high}_{\Bdot^{d/2}_{2,1}}
		\right).
		\label{eq:high-Besov-critical}
	\end{align}
\end{proposition}

\begin{proof}
	We work first with the symmetrized variables $Y$ from \eqref{eq:cstar-Y}.  Set
	\[
	V_1=Y_1,\qquad V_2=rY_2,\qquad V=(V_1,V_2)^\top.
	\]
	Since $\partial_t\widehat Y=A(r)\widehat Y$, the vector $V$ solves
	\begin{equation*} 
		\partial_t\widehat V=A_h(r)\widehat V,
		\qquad
		A_h(r)=
		\begin{pmatrix}
			-\nu r^2&c_*\\
			-c_*r^2&0
		\end{pmatrix}.
	\end{equation*}
	The  corresponding eigenvalues are $\lambda_\pm(r)=(-\nu r^2\pm\sqrt{\nu^2r^4-4c_*^2r^2})/2$ and
	\begin{equation*} 
		\lambda_-(r)=-\nu r^2+\frac{c_*^2}{\nu}+O(r^{-2}),
		\qquad
		\lambda_+(r)=-\frac{c_*^2}{\nu}+O(r^{-2})
		\qquad(r\to\infty).
	\end{equation*}
	
	Set
	\[
	d(r):=\lambda_+(r)-\lambda_-(r)
	=\sqrt{\nu^2r^4-4c_*^2r^2}.
	\]
	Choose $r_1>2c_*/\nu$ so large that, for all $r\ge r_1$,
	\begin{equation}\label{eq:high-r-coer}
		d(r)\ge \frac{\nu}{2}r^2,
		\qquad -\lambda_+(r)\ge \gamma_0,
		\qquad -\lambda_-(r)\ge \gamma_1r^2
	\end{equation}
	for some $\gamma_0,\gamma_1>0$.  
	 The semigroup representation gives  
	\[
	\ee^{tA_h(r)}=\ee^{\lambda_+(r)t}\Pi_+(r)
	+\ee^{\lambda_-(r)t}\Pi_-(r),
	\]
	where
	\begin{equation*} 
		\Pi_+(r)=\frac{A_h(r)-\lambda_-(r)I}{d(r)}
		=\frac1{d(r)}
		\begin{pmatrix}
			\lambda_+(r)&c_*\\
			-c_*r^2&-\lambda_-(r)
		\end{pmatrix},
		\qquad
		\Pi_-(r)=I-\Pi_+(r).
	\end{equation*}
		Then direct calculation yields, for $r\ge r_1$,
	\begin{align}
		\abs{\widehat V(t,\xi)}
		&\le C\left(\ee^{-\gamma_0t}+\ee^{-\gamma_1r^2t}\right)
		\abs{\widehat V_0(\xi)},
		\label{eq:V-high-p}\\
		r^2\abs{\widehat V_1(t,\xi)}
		&\le C\left(\ee^{-\gamma_0t}
		+r^2\ee^{-\gamma_1r^2t}\right)
		\abs{\widehat V_0(\xi)}.
		\label{eq:V1-high-s}
	\end{align}
	Fix $0<c<\gamma_0/2$ and then increase $r_1$ so that
	$c<\gamma_1r_1^2/2$.  Employing 
	\eqref{eq:V-high-p}-\eqref{eq:V1-high-s} implies
	\begin{align}
		&\sup_{t\ge0}\ee^{ct}\abs{\widehat V(t,\xi)}
		+\int_0^\infty\ee^{ct}
		\left(r^2\abs{\widehat V_1(t,\xi)}
		+\abs{\widehat V_2(t,\xi)}\right)\dd t
		\notag\\
		&\qquad\le C\abs{\widehat V_0(\xi)}
		\qquad(r\ge r_1).
		\label{6363446}
	\end{align}
	
	Choose $j_1$ such that every annulus supporting $\dyad_j$, $j\ge j_1$, lies in
	$\{r\ge r_1\}$.  On that annulus, $r\simeq2^j$, while the constant
	symmetrization \eqref{eq:cstar-Y} gives
	\[
	\abs V\simeq \abs u+r\abs z.
	\]
	Applying Plancherel's theorem to \eqref{6363446} and Minkowski's
	inequality for the time integral, we obtain, for all $j\ge j_1$,
	\begin{align*}
		&\sup_{t\ge0}\ee^{ct}
		\left(\norm{\dyad_j u_L(t)}_{2}+2^j\norm{\dyad_jz_L(t)}_{2}\right)
		+
		\int_0^\infty \ee^{ct}
		\left(2^{2j}\norm{\dyad_j u_L(t)}_{2}
		+2^j\norm{\dyad_jz_L(t)}_{2}\right)\dd t
		\notag\\
		&\qquad\le C\left(\norm{\dyad_j u_0}_{2}
		+2^j\norm{\dyad_jz_0}_{2}\right).
	\end{align*}
	
	Finally, we multiply the above inequality by $2^{j(d/2-1)}$ and sum over $j\ge j_1$. On the finitely many shells $j_0-1\le j<j_1$ the frequency variable ranges in a compact annulus on which both eigenvalues of $A_h(r)$ (equivalently of $A(r)$) have negative real parts and depend continuously on $r$; hence the semigroup is uniformly exponentially stable there, and these shells are absorbed into the constant. This completes the proof of Proposition \ref{prop:high-frequency}.
\end{proof}

 Let $(u,z)$ solve the system \eqref{eq:forced-linear}, then we can rewrite the \eqref{eq:forced-linear} by
\begin{equation}\label{eq:forced-p2}
	\partial_tu+\nu\Lambda^2u-\beta_1 q=F,
	\qquad
	\partial_tq+\beta_2\Lambda^2 u=G_q,
\end{equation}
with
\begin{equation*}  
 \beta_1:=\nu_1,\quad\beta_2:=\frac{\nu_2}{2},\qquad  G_q:=\Lambda G.
\end{equation*}

\begin{proposition}\label{6376976}
	Let $U=(u,z)$ solve \eqref{eq:forced-linear}, so that $(u,q)$, $q=\Lambda z$, solves \eqref{eq:forced-p2}, and
	assume
	\begin{align*}
		&u_0^\high\in\Bdot^{d/2-1}_{2,1},\qquad
		z_0^\high\in\Bdot^{d/2}_{2,1},\\
		&F^\high\in L^1(0,T;\Bdot^{d/2-1}_{2,1}),\qquad
		G^\high\in L^1(0,T;\Bdot^{d/2}_{2,1}).
	\end{align*}
	Then there exists some $C>0$  such that
	\begin{align}
		&\norm{u^\high}_{\CL^\infty_T(\Bdot^{d/2-1}_{2,1})}
		+\norm{z^\high}_{\CL^\infty_T(\Bdot^{d/2}_{2,1})}
		+\norm{u^\high}_{L^1_T(\Bdot^{d/2+1}_{2,1})}
		+\norm{z^\high}_{L^1_T(\Bdot^{d/2}_{2,1})}
		\notag\\
		&\qquad\le C\Bigl(
		\norm{u_0^\high}_{\Bdot^{d/2-1}_{2,1}}
		+\norm{z_0^\high}_{\Bdot^{d/2}_{2,1}}
		+\norm{F^\high}_{L^1_T(\Bdot^{d/2-1}_{2,1})}
		+\norm{G^\high}_{L^1_T(\Bdot^{d/2}_{2,1})}
		\Bigr).
		\label{eq:f-h-p2-h}
	\end{align}
	Equivalently, in terms of $q=\Lambda z$,
	\begin{align}
		&\norm{u^\high}_{\CL^\infty_T(\Bdot^{d/2-1}_{2,1})}
		+\norm{q^\high}_{\CL^\infty_T(\Bdot^{d/2-1}_{2,1})}
		+\norm{u^\high}_{L^1_T(\Bdot^{d/2+1}_{2,1})}
		+\norm{q^\high}_{L^1_T(\Bdot^{d/2-1}_{2,1})}
		\notag\\
		&\qquad\le C\Bigl(
		\norm{u_0^\high}_{\Bdot^{d/2-1}_{2,1}}
		+\norm{q_0^\high}_{\Bdot^{d/2-1}_{2,1}}
		+\norm{F^\high}_{L^1_T(\Bdot^{d/2-1}_{2,1})}
		+\norm{G_q^\high}_{L^1_T(\Bdot^{d/2-1}_{2,1})}
		\Bigr).
		\label{526856}
	\end{align}
	
	
\end{proposition}

\begin{proof}

 Let $u_j=\dyad_ju$ and
	$q_j=\dyad_jq$.  Choose
	\begin{equation*} 
		0<\kappa_1<\frac{\nu}{2\beta_2},
		\qquad b_1=\frac{\kappa_1\nu}{\beta_2},
	\end{equation*}
	and define
	\begin{equation}\label{526436}
		\Theta_j^2
		:=\norm{u_j}_2^2+b_1\norm{q_j}_2^2-2\kappa_1\ip{u_j}{q_j}.
	\end{equation}
	Noticing the fact $b_1>2\kappa_1^2$, one has  
	\begin{equation}\label{52643632}
		c(\norm{u_j}_2+\norm{q_j}_2)
		\le \Theta_j
		\le C(\norm{u_j}_2+\norm{q_j}_2).
	\end{equation}
	A direct differentiation of \eqref{526436}, using
	$b_1 \beta_2=\kappa_1\nu$, yields
	\begin{align}
		\frac12\frac{\dd}{\dd t}\Theta_j^2
		={}&-(\nu-\kappa_1\beta_2)\norm{\Lambda u_j}_2^2
		-\kappa_1\beta_1\norm{q_j}_2^2
		+\beta_1\ip{u_j}{q_j}
		\notag\\
		&+\ip{F_j}{u_j-\kappa_1 q_j}
		+\ip{(G_q)_j}{b_1q_j-\kappa_1 u_j}.
		\label{5264321}
	\end{align}
	By Bernstein's inequality on an annulus,
	$\norm{u_j}_2\le C2^{-j}\norm{\Lambda u_j}_2$.  Hence, for any
	$\varepsilon>0$,
	\begin{equation*} 
		\beta_1\abs{\ip{u_j}{q_j}}
		\le \varepsilon\norm{\Lambda u_j}_2^2
		+C_\varepsilon\beta_1^22^{-2j}\norm{q_j}_2^2.
	\end{equation*}
	
	Choose first
	$\varepsilon=(\nu-\kappa_1\beta_2)/4$ and then $j_h$ so large that
	$C_\varepsilon\beta_1^22^{-2j_h}\le\kappa_1\beta_1/4$.  For all
	$j\ge j_h$, \eqref{5264321} therefore gives, 
	\begin{equation}\label{525666}
		\frac{\dd}{\dd t}\Theta_j
		+c\bigl(\norm{u_j}_2+\norm{q_j}_2\bigr)
		\le C\bigl(\norm{F_j}_2+\norm{(G_q)_j}_2\bigr).
	\end{equation}
	
	The $2^{2j}$ smoothing of $u_j$ is not contained in
	\eqref{525666} and must be recovered separately.  Taking the
	$L^2$ scalar product of the first equation {in \eqref{eq:forced-p2}} with
	$u_j$, applying Bernstein and Young inequalities gives
	\begin{equation}\label{636885}
		\frac{\dd}{\dd t}\norm{u_j}_2
		+c2^{2j}\norm{u_j}_2
		\le C\norm{q_j}_2+C\norm{F_j}_2.
	\end{equation}
	
	Choose $0<\eta_1\ll1$ so that the $q_j$ term generated by
	$\eta_1$ times \eqref{636885} is absorbed by
	\eqref{525666}, and set
	\begin{equation*}
		H_j:=\Theta_j+\eta_1\norm{u_j}_2.
	\end{equation*}
	Then
	\begin{equation}\label{638444}
		\frac{\dd}{\dd t}H_j
		+c\left(2^{2j}\norm{u_j}_2+\norm{q_j}_2\right)
		\le C\left(\norm{F_j}_2+\norm{(G_q)_j}_2\right),
		\qquad j\ge j_h,
	\end{equation}
	and $H_j\simeq\norm{u_j}_2+\norm{q_j}_2$.
	
	It remains to treat the finitely many intermediate shells $j_0-1\le j<j_h$. On these shells $2^{2j}$ is bounded above and below by positive constants, so that $2^{2j}\norm{u_j}_2\simeq\norm{u_j}_2$, and the required damped block inequality follows from Duhamel's formula together with the uniform exponential stability of the semigroup on compact annuli noted at the end of the proof of Proposition~\ref{prop:high-frequency}. Thus \eqref{638444} is valid for all $j\geq j_0-1$ (with $H_j$ replaced by a suitable Lyapunov functional equivalent to the  $\norm{u_j}_2+\norm{q_j}_2$ on the intermediate shells). Integrating the block inequality on $[0,T]$, multiplying by
	$2^{j(d/2-1)}$, and summing over $j\ge j_0-1$, we obtain
	 \eqref{526856}. 
	  Finally, on all high annulus, we have
	\begin{equation*} 
		2^{j(d/2-1)}\norm{q_j}_2
		\simeq 2^{jd/2}\norm{z_j}_2,
		\qquad
		2^{j(d/2-1)}\norm{(G_q)_j}_2
		\simeq 2^{jd/2}\norm{G_j}_2.
	\end{equation*}
	This proves \eqref{eq:f-h-p2-h} and thus the proof of Proposition \ref{6376976} is completed.
\end{proof}
\subsection{Necessary and sufficient condition for upper decay}

The following theorem characterizes the upper decay of the linear active system \eqref{eq:l-1}.
\begin{theorem}\label{thm:li-upp-}
	Let
	\begin{equation}\label{eq:linear-r-p2}
		\sigma_1<\sigma\le \frac d2-1,
	\end{equation}
	and $U_L=(u_L,z_L)$ solve \eqref{eq:l-1}. Assume
	\begin{equation}\label{eq:li-upp-b}
		{U_0^\low\in\Bdot^{\sigma}_{2,1},}
		\qquad
		u_0^\high\in\Bdot^{d/2-1}_{2,1},
		\qquad
		z_0^\high\in\Bdot^{d/2}_{2,1}.
	\end{equation}
	Then the following statements are equivalent:
	\begin{enumerate}[label=\textup{(\roman*)}]
		\item {$U_0^\low\in\Bdot^{\sigma_1}_{2,\infty}$;}
		\item there exists $C>0$ such that the full hybrid estimate  holds
		\begin{equation}\label{eq:linear-upp-h}
			\norm{U_L(t)}_{\cH^{\sigma}}
			\le C\braket t^{-\frac12(\sigma-\sigma_1)},
			\qquad t\ge0;
		\end{equation}
	\item  there exists $C>0$ such that the low-frequency component satisfies
	\begin{equation}\label{eq:linear-upp-r}
		\norm{U_L^\low(t)}_{{\Bdot^{\sigma}_{2,1}}}
		\le C\braket t^{-\frac12(\sigma-\sigma_1)},	\qquad t\ge0.
	\end{equation}
	\end{enumerate}
\end{theorem}

\begin{proof}
	$(i) \Rightarrow (iii)$. 
	Assume first that { $U_0^\low\in\Bdot^{\sigma_1}_{2,\infty}$}. By the upper estimate in \eqref{eq:d-two-s},
	\begin{align}
		{ \norm{U_L^\low(t)}_{\Bdot^{\sigma}_{2,1}}}
		&=\sum_{j\le j_0}2^{j\sigma}\norm{\dyad_jU_L(t)}_2\notag\\
		&\le C\sum_{j\le j_0}2^{j\sigma}\ee^{-c2^{2j}t}\norm{\dyad_jU_0}_2\notag\\
		&\le C\norm{U_0^\low}_{\Bdot^{\sigma_1}_{2,\infty}}
		\sum_{j\le j_0}2^{j(\sigma-\sigma_1)}\ee^{-c2^{2j}t}.
		\label{eq:l-upp-sum}
	\end{align}
	Recall that $\alpha=\frac{\sigma-\sigma_1}{2}>0$. For $t\ge1$,
	\[
	2^{2j\alpha}\ee^{-c2^{2j}t}
	=t^{-\alpha}(2^j\sqrt t)^{2\alpha}\ee^{-c(2^j\sqrt t)^2}.
	\]
	Splitting the dyadic sum at the index for which $2^j\sqrt t$ is
	comparable to one gives a geometric sum below that index and an
	exponentially convergent sum above it. In particular,
	\begin{equation*}
		\sum_{j\in\Z}(2^j\sqrt t)^{2\alpha}\ee^{-c(2^j\sqrt t)^2}\le C_\alpha,
	\end{equation*}
	so \eqref{eq:l-upp-sum} gives
	{\[
	\norm{U_L^\low(t)}_{\Bdot^{\sigma}_{2,1}}
	\le C t^{-\alpha}\norm{U_0^\low}_{\Bdot^{\sigma_1}_{2,\infty}}
	\qquad(t\ge1).
	\]}
	
	For $0\le t\le1$, the semigroup is bounded on the fixed low-frequency Besov space and $\braket t\simeq1$; hence
	{\[
	\norm{U_L^\low(t)}_{\Bdot^{\sigma}_{2,1}}
	\le C\norm{U_0^\low}_{\Bdot^{\sigma}_{2,1}}.
	\]}
	This proves \eqref{eq:linear-upp-r}.
	
$(iii) \Rightarrow (i)$. 	Conversely, assume \eqref{eq:linear-upp-r}. Fix $j\le j_0$ sufficiently negative and set
	\begin{equation*} 
		t_j=2^{-2j}.
	\end{equation*}
	Then $t_j\ge1$ and $2^{2j}t_j=1$. By the lower dyadic estimate,
	\begin{equation}\label{eq:lower-at-tj}
		\norm{\dyad_jU_L(t_j)}_2
		\ge c\ee^{-C}\norm{\dyad_jU_0}_2.
	\end{equation}
	
	On the other hand,
	\begin{align*}
		2^{j\sigma}\norm{\dyad_jU_L(t_j)}_2
		&\le\norm{U_L^\low(t_j)}_{{\Bdot^{\sigma}_{2,1}}}\le C t_j^{-\frac12(\sigma-\sigma_1)}
		=C2^{j(\sigma-\sigma_1)}.
	\end{align*}
	Combining this with \eqref{eq:lower-at-tj} gives
	\begin{equation*} 
		2^{j\sigma_1}\norm{\dyad_jU_0}_2\le C
	\end{equation*}
	for any sufficiently negative $j$. Choose a fixed lower endpoint $J$ for the remaining finite family
	$J\le j\le j_0$. These blocks are controlled by
	\eqref{eq:li-upp-b}; indeed,
	\[
	2^{j\sigma_1}\norm{\dyad_jU_0}_2
	=2^{-j(\sigma-\sigma_1)}2^{j\sigma}\norm{\dyad_jU_0}_2
	\le C(J,j_0)\norm{U_0^\low}_{{\Bdot^{\sigma}_{2,1}}}.
	\]
	Thus ${U_0^\low\in\Bdot^{\sigma_1}_{2,\infty}}$.

	$(ii) \iff (iii)$. By \eqref{eq:high-Besov-critical}, the high-frequency estimates satisfy 
	\begin{equation}\label{8574555}
		\norm{u_L^\high(t)}_{\Bdot^{d/2-1}_{2,1}}
		+\norm{z_L^\high(t)}_{\Bdot^{d/2}_{2,1}}
		\le C\ee^{-ct}
		\left(
		\norm{u_0^\high}_{\Bdot^{d/2-1}_{2,1}}
		+\norm{z_0^\high}_{\Bdot^{d/2}_{2,1}}
		\right).
	\end{equation}
	
	Hence the low-frequency estimate implies \eqref{eq:linear-upp-h}. Conversely, the
	low-frequency seminorm is one of the nonnegative terms in $\norm{U_L}_{\cH^{\sigma}}$,
	so the hybrid estimate implies \eqref{eq:linear-upp-r}.  Thus, we complete the proof of Theorem \ref{thm:li-upp-}.  	 
\end{proof}
\subsection{Two-sided decay}

\begin{theorem}\label{thm:linear-t-s}
	
	Let $U_L=(u_L,z_L)$ solve \eqref{eq:l-1}. Assume \eqref{eq:linear-r-p2} and \eqref{eq:li-upp-b} hold. Then the following statements are equivalent:
	\begin{enumerate}[label=\textup{(\roman*)}]
		\item $U_0^\low\in\frakB^{\sigma_1}_{2,\infty}$;
		\item there exist constants $c,C>0$ and $t_0\ge1$ such that
		\begin{equation}\label{6373573}
			c t^{-\frac12(\sigma-\sigma_1)}
			\le\norm{U_L(t)}_{\cH^{\sigma}}
			\le C t^{-\frac12(\sigma-\sigma_1)},
			\qquad t\ge t_0;
		\end{equation}
		\item there exist constants $c,C>0$ and $t_0\ge1$ such that the low-frequency two-sided estimates hold
		\begin{equation}\label{673634}
			c t^{-\frac12(\sigma-\sigma_1)}
			\le\norm{U_L^\low(t)}_{ \Bdot^{\sigma}_{2,1}}
			\le C t^{-\frac12(\sigma-\sigma_1)},
			\qquad t\ge t_0.
		\end{equation}
	\end{enumerate}
\end{theorem}
\begin{proof}
$(i) \Rightarrow (iii)$. 	Assume first that $U_0^\low\in\frakB^{\sigma_1}_{2,\infty}$, and recall that $(j_k)$, $M$, and $c_0$ are given in Definition \ref{def:frakB}.  The upper bound follows from \cref{thm:li-upp-}. Given a sufficiently large $t$, choose $j_k$ to be the largest element of the sequence satisfying
	\begin{equation}\label{eq:choose-jk}
		2^{j_k}\le t^{-1/2}.
	\end{equation}
	
  If $2^{j_k}\le2^{-M}t^{-1/2}$, then the preceding element $j_{k-1}$ satisfies $j_{k-1}\le j_k+M$ and hence $2^{j_{k-1}}\le t^{-1/2}$. Hence
	\begin{equation}\label{eq:jk-comparable}
		2^{-M}t^{-1/2}<2^{j_k}\le t^{-1/2}.
	\end{equation}
	Using a single dyadic block and the lower semigroup estimate,
	\begin{align*}
		\norm{U_L^\low(t)}_{\Bdot^{\sigma}_{2,1}}
		&\ge2^{j_k\sigma}\norm{\dyad_{j_k}U_L(t)}_2\\
		&\ge c2^{j_k\sigma}\ee^{-C2^{2j_k}t}\norm{\dyad_{j_k}U_0}_2\\
		&\ge cc_0\ee^{-C}2^{j_k(\sigma-\sigma_1)}\\
		&\ge c' t^{-\frac12(\sigma-\sigma_1)},
	\end{align*}
	where \eqref{eq:jk-comparable} was used in the last step. This then proves \eqref{673634}.
	
	$(iii) \Rightarrow (i)$.   Conversely, assume \eqref{673634}. By   \cref{thm:li-upp-},
	\[
	A:=\sup_{j\le j_0}a_j<\infty,
	\qquad
	a_j:=2^{j\sigma_1}\norm{\dyad_jU_0}_2.
	\]
	For $t\ge t_0$, choose $m=m(t)\in\Z$ such that
	
	\begin{equation}\label{eq:m-of-t}
		2^m\le t^{-1/2}<2^{m+1}.
	\end{equation}
	The upper dyadic semigroup bound gives
	\begin{align}
		t^{\frac12(\sigma-\sigma_1)}
		\norm{U_L^\low(t)}_{{\Bdot^{\sigma}_{2,1}}}
		&\le C\sum_{j\le j_0}
		(t^{1/2}2^j)^{\sigma-\sigma_1}
		\ee^{-c2^{2j}t}a_j\notag\\
		&\le C\sum_{j\le j_0}\omega(j-m)a_j,
		\label{eq:w-a-sum}
	\end{align}
	where
	\begin{equation*} 
		\omega(k):=2^{k(\sigma-\sigma_1)}\ee^{-c_12^{2k}}
	\end{equation*}
	for a possibly smaller $c_1>0$. The sequence $\omega$ belongs to $\ell^1(\Z)$.
	
	Let $\underline c$ denote the positive lower constant in \eqref{673634}. Choose $J\ge1$ so large that
	\begin{equation}\label{eq:omega-tail}
		CA\sum_{|k|>J}\omega(k)\le\frac{\underline c}{2}.
	\end{equation}
	Combining the lower bound in \eqref{673634} with \eqref{eq:w-a-sum} and \eqref{eq:omega-tail} yields
	\begin{equation*} 
		\sum_{|j-m|\le J}\omega(j-m)a_j\ge\frac{\underline c}{2C}.
	\end{equation*}
	
	Since there are at most $2J+1$ terms and $\omega$ is bounded, there exists an index $j=j(m)$ such that
	\begin{equation}\label{eq:j-near-m}
		|j-m|\le J,
		\qquad
		a_j\ge c_2>0,
	\end{equation}
	where $c_2$ is independent of $m$.	Now take
	\[
	m_n=-(2J+2)n
	\]
	for all sufficiently large $n$, and set $t_n=2^{-2m_n}$. Then \eqref{eq:m-of-t} holds with $m=m_n$. Choose $j_n$ by \eqref{eq:j-near-m}. We have $j_n\to-\infty$, $a_{j_n}\ge c_2$, and
	\begin{align*}
		|j_{n+1}-j_n|
		&\le|j_{n+1}-m_{n+1}|+|m_{n+1}-m_n|+|m_n-j_n|\\
		&\le J+(2J+2)+J=4J+2.
	\end{align*}
	Moreover, because $m_{n+1}\le m_n-(2J+2)$ and $|j_n-m_n|\le J$, the sequence is strictly decreasing. Hence $U_0^\low\in\frakB^{\sigma_1}_{2,\infty}$.
	
	$(ii) \iff (iii)$.  	It remains to justify the hybrid statement.  If the dyadic nondegeneracy holds, the
	low-frequency lower bound is already a lower bound for the hybrid norm, while
	\eqref{8574555} gives the hybrid upper bound.  Conversely, suppose
	the hybrid norm has the two-sided algebraic behavior, the hybrid upper bound implies the
	low-frequency upper bound.  Moreover, by \eqref{8574555},
	\begin{align*}
		\norm{U_L^\low(t)}_{{\Bdot^{\sigma}_{2,1}}}
		&\ge \norm{U_L(t)}_{\cH^{\sigma}}
		-\norm{u_L^\high(t)}_{\Bdot^{d/2-1}_{2,1}}
		-\norm{z_L^\high(t)}_{\Bdot^{d/2}_{2,1}}\\
		&\ge ct^{-\frac12(\sigma-\sigma_1)}-C\ee^{-c_1t}.
	\end{align*}
	For all sufficiently large $t$, the last expression is bounded below by
	$(c/2)t^{-(\sigma-\sigma_1)/2}$.  Hence the low-frequency norm also has two-sided
	decay.
\end{proof}

\section{{Critical global solutions and the dissipative quotient}}\label{sec:kernel}

The results  in Theorem \ref{thm:known-global} imply that the whole stress can be  controlled  in the energy norm, but its dissipative part is the pair $(u,z)$. 
Corresponding to the linear system \eqref{eq:l-1}, the nonlinear system is given by
\begin{equation}\label{eq:nonlinear-1}
\begin{cases}
\partial_t u+\mathbb P[(u\cdot\nabla)u]+\nu\Lambda^2u-\nu_1\Lambda z=0,\\
\partial_t z+\cT[(u\cdot\nabla)\tau]
+\cT Q(\tau,\nabla u)
+\dfrac{\nu_2}{2}\Lambda u=0,\\
\diver u=0,\\
\mathbb P\diver\tau=\Lambda z.
\end{cases}
\end{equation}

The following identity makes the derivative structure of $\cT(u\cdot\nabla\tau)$ explicit, which will be used in the next proposition. 
\begin{lemma}[{ \cite{ChenHao2019,Zhu2018}}]\label{lem:chen-hao}
Let $u$ be divergence-free and $\tau$ be smooth. Then
\begin{equation*} 
\begin{aligned}
\Pj\diver(u\cdot\nabla\tau)
={}&\Pj\bigl(u\cdot\nabla\Pj\diver\tau\bigr)
+\Pj(\nabla u\cdot\nabla\tau) -\Pj\bigl(\nabla u\cdot\nabla\Delta^{-1}\diver\diver\tau\bigr),
\end{aligned}
\end{equation*}

{and
\[
[\nabla u\cdot\nabla\tau]_i
=\sum_j\partial_ju\cdot\nabla\tau_{ij},
\qquad
[\nabla u\cdot\nabla\Delta^{-1}\diver\diver\tau]_i
=\partial_i u\cdot\nabla\Delta^{-1}\diver\diver\tau.
\]}
\end{lemma} 

 We now justify the inequality for the active pair.  
\begin{proposition}[Active critical Lyapunov inequality]\label{prop:act-Lyap}

Under the assumptions of \cref{thm:known-global}, define
\begin{equation}\label{eq:active-11}
\cE(t)=
\norm{(u,z)^\low(t)}_{\Bdot^{d/2-1}_{2,1}}
+\norm{u^\high(t)}_{\Bdot^{d/2-1}_{2,1}}
+\norm{z^\high(t)}_{\Bdot^{d/2}_{2,1}}
\end{equation}
and
\begin{equation}\label{eq:active-d-11}
\cD(t)=
\norm{(u,z)^\low(t)}_{\Bdot^{d/2+1}_{2,1}}
+\norm{u^\high(t)}_{\Bdot^{d/2+1}_{2,1}}
+\norm{z^\high(t)}_{\Bdot^{d/2}_{2,1}}.
\end{equation}
Set
\begin{equation}\label{eq:criti-11}
\cA(t):=
\norm{(u,\tau)^\low(t)}_{\Bdot^{d/2-1}_{2,1}}
+\norm{u^\high(t)}_{\Bdot^{d/2-1}_{2,1}}
+\norm{\tau^\high(t)}_{\Bdot^{d/2}_{2,1}}.
\end{equation}
There exist energy
$\mathscr E_a$  and dissipation
$\mathscr D_a$ such that
\begin{equation}\label{eq:Lyap-equi}
C^{-1}\cE(t)\le\mathscr E_a(t)\le C\cE(t),
\qquad
C^{-1}\cD(t)\le\mathscr D_a(t)\le C\cD(t),
\end{equation}
and
\begin{equation}\label{52523}
\frac{\dd}{\dd t}\mathscr E_a(t)+c_0\mathscr D_a(t)
\le C\cA(t)\mathscr D_a(t)
\end{equation}
for some $c_0, C>0$, $t>0$.  Moreover, we have 
\begin{equation}\label{525232}
{\frac{\dd}{\dd t}\mathscr E_a(t)+\frac{c_0}{2}  \mathscr D_a(t)\le0.}
\end{equation}
for small  $\varepsilon_0>0$ in \eqref{eq:X0}.
\end{proposition}

\begin{proof}

We first construct the linear dyadic energies and then estimate the nonlinear remainders.

\smallskip
\noindent\emph{ Step 1: Estimates for low-frequency.}
For $j\le j_0$, write $u_j=\dyad_ju$, $z_j=\dyad_jz$.  After localization to \eqref{eq:nonlinear-1},  the active
system has the form
\begin{equation}\label{5252432}
\partial_tu_j+\nu\Lambda^2u_j-\beta_1\Lambda z_j={F_j},
\qquad
\partial_tz_j+\beta_2\Lambda u_j={G_j},
\end{equation}
where
\begin{align*}
	F_j:=-\dot{\Delta}_j(\mathbb P(u\cdot\nabla u)),~~G_j:=\dot{\Delta}_j(\Lambda^{-1}\mathbb P \diver(-u\cdot \nabla\tau-Q(\tau,\nabla u))).
\end{align*}

Choose $0<\kappa_2<\nu$ and define
\begin{equation}\label{5232344}
L_j^2:=\beta_2\norm{u_j}_2^2+\beta_1\norm{z_j}_2^2
-2\kappa_2\ip{\Lambda u_j}{z_j}.
\end{equation}
By Bernstein's inequality 
\[
2\kappa_2\abs{\ip{\Lambda u_j}{z_j}}
\le C\kappa_22^j(\norm{u_j}_2^2+\norm{z_j}_2^2).
\]
Taking $j_0$ more negative if necessary gives
\begin{equation}\label{523239}
c(\norm{u_j}_2+\norm{z_j}_2)\le L_j
\le C(\norm{u_j}_2+\norm{z_j}_2).
\end{equation}
Multiplying the two equations in \eqref{5252432} by
$\beta_2 u_j$ and $\beta_1 z_j$, respectively,  yields
\begin{equation}\label{5232139}
\frac12\frac{\dd}{\dd t}
(\beta_2\norm{u_j}_2^2+\beta_1\norm{z_j}_2^2)
+\beta_2\nu\norm{\Lambda u_j}_2^2
=\beta_2\ip{F_j}{u_j}+\beta_1\ip{G_j}{z_j}.
\end{equation}

Moreover,
\begin{align}
\frac{\dd}{\dd t}\ip{\Lambda u_j}{z_j}
={}&-\nu\ip{\Lambda^2u_j}{\Lambda z_j}
+\beta_1\norm{\Lambda z_j}_2^2
-\beta_2\norm{\Lambda u_j}_2^2\notag\\
&+\ip{\Lambda F_j}{z_j}+\ip{\Lambda u_j}{G_j}.
\label{52132139}
\end{align}
Subtracting $\kappa_2$ times \eqref{52132139} from
\eqref{5232139}, and using
\[
\abs{\ip{\Lambda^2u_j}{\Lambda z_j}}
\le C2^j\norm{\Lambda u_j}_2\norm{\Lambda z_j}_2,
\]
we then obtain from \eqref{523239}
\begin{equation*} 
\frac12\frac{\dd}{\dd t}L_j^2
+c2^{2j}(\norm{u_j}_2^2+\norm{z_j}_2^2)
\le C L_j(\norm{F_j}_2+\norm{G_j}_2).
\end{equation*}
   Dividing by a standard regularization $(L_j^2+\varepsilon^2)^{1/2}$ and then letting $\varepsilon$ small enough, we have
\begin{equation}\label{53764}
\frac{\dd}{\dd t}L_j+c2^{2j}L_j
\le C(\norm{F_j}_2+\norm{G_j}_2).
\end{equation}

\smallskip
\noindent\emph{Step 2: Estimates for high-frequency.}
For $j\ge j_0-1$, use $q_j=\dyad_jq$ and consider
\begin{equation}\label{52466}
\partial_tu_j+\nu\Lambda^2u_j-\beta_1 q_j={\mathcal F_j},
\qquad
\partial_tq_j+\beta_2\Lambda^2u_j={\mathcal G_j}.
\end{equation}
where
\begin{align*}
	\mathcal F_j:=F_j,~~\mathcal G_j=\Lambda G_j.
\end{align*}
Fix $0<\kappa_3<\nu/(2\beta_2)$ and put $b_2=\kappa_3\nu/\beta_2$.  Define  
\begin{equation}\label{5372164}
\Theta_j^2:=\norm{u_j}_2^2+b_2\norm{q_j}_2^2
-2\kappa_3\ip{u_j}{q_j}.
\end{equation}
Since $b_2=\kappa_3\nu/\beta_2>2\kappa_3^2$, after decreasing $\kappa_3$ if necessary,
\begin{equation}\label{53721649}
c(\norm{u_j}_2+\norm{q_j}_2)\le \Theta_j
\le C(\norm{u_j}_2+\norm{q_j}_2).
\end{equation}
A direct differentiation gives
\begin{align*}
\frac12\frac{\dd}{\dd t}\Theta_j^2
={}&-(\nu-\kappa_3\beta_2)\norm{\Lambda u_j}_2^2
-\kappa_3\beta_1\norm{q_j}_2^2
+\beta_1\ip{u_j}{q_j}\notag\\
&+\ip{\mathcal F_j}{u_j-\kappa_3q_j}
+\ip{\mathcal G_j}{b_2q_j-\kappa_3u_j}.
\label{eq:high-energy-raw}
\end{align*}

Choose $j_h$ so large that, for all $j\ge j_h$, Bernstein and Young inequalities give
\[
\beta_1\abs{\ip{u_j}{q_j}}
\le\frac{\nu-\kappa_3\beta_2}{2}\norm{\Lambda u_j}_2^2
+\frac{\kappa_3\beta_1}{2}\norm{q_j}_2^2.
\]
Regularizing $\Theta_j$ before division as in the low-frequency argument, we obtain
\begin{equation}\label{575329}
\frac{\dd}{\dd t}\Theta_j
+c\bigl(\norm{u_j}_2+\norm{q_j}_2\bigr)
\le C(\norm{\mathcal F_j}_2+\norm{\mathcal G_j}_2).
\end{equation}
 Taking its scalar product with $u_j$, using Bernstein's
inequality, and dividing by $(\norm{u_j}_2^2+\varepsilon^2)^{1/2}$ before letting
$\varepsilon$ small enough, yields
\begin{equation}\label{85445}
\frac{\dd}{\dd t}\norm{u_j}_2+c2^{2j}\norm{u_j}_2
\le C\norm{q_j}_2+C\norm{\mathcal F_j}_2.
\end{equation}

Choose $0<\eta_2\ll1$, so that the term $C\eta_2\norm{q_j}_2$ is absorbed by the
$q_j$-damping in \eqref{575329}, and set
\begin{equation}\label{57522329}
H_j:=\Theta_j+\eta_2\norm{u_j}_2,
\end{equation}
we have
\begin{equation}\label{8544541}
c(\norm{u_j}_2+\norm{q_j}_2)\le H_j
\le C(\norm{u_j}_2+\norm{q_j}_2),
\end{equation}
then $\eta_2$ times
\eqref{85445} and add  \eqref{575329}  give
\begin{equation}\label{854454}
\frac{\dd}{\dd t}H_j
+c\left(2^{2j}\norm{u_j}_2+\norm{q_j}_2\right)
\le C(\norm{\mathcal F_j}_2+\norm{\mathcal G_j}_2).
\end{equation}

The low threshold $j_0$ fixed by \eqref{eq:j0-low} and the high threshold $j_h$ need
not coincide, but there are only finitely many intermediate shells $j_0<j<j_h$. On these shells we use the low-frequency functional $L_j$ of Step~1. Indeed, inspecting Step~1, the only places where $j\le j_0$ was used are the equivalence \eqref{523239} and the absorption of the cross term $\kappa_2\nu\ip{\Lambda^2u_j}{\Lambda z_j}$; both remain valid for all $j\le j_h$ provided $\kappa_2$ is chosen small enough in terms of $2^{j_h}$, and since $j_h$ depends only on $\nu,\beta_1,\beta_2$ this is a legitimate choice. Hence \eqref{53764} holds for all $j\le j_h$, with $2^{2j}\simeq1$ on the intermediate shells.
  
\smallskip
\noindent\emph{Step 3: Nonlinear localization and cancellation of principal transport.}
We now identify the forcing terms in systems
\eqref{5252432} and \eqref{52466}.
For $j\le j_0$, localizing the { equations
\eqref{eq:nonlinear-1}} give
{\begin{equation}\label{96456}
\left\{
\begin{aligned}
\partial_tu_j+\mathbb P(u\cdot\nabla u_j)+\nu\Lambda^2u_j-\beta_1\Lambda z_j
&=F_j^{\rm rem},\\
\partial_tz_j+u\cdot\nabla z_j+\beta_2\Lambda u_j
&=G_j^{\rm rem},
\end{aligned}
\right.
\end{equation}
where
\begin{align*}
F_j^{\rm rem}
&:=-\mathbb P[\dyad_j,u\cdot\nabla]u,
\\
G_j^{\rm rem}
&:=-[\dyad_j,u\cdot\nabla]z
-\dyad_j[\cT,u\cdot\nabla]\tau
-\dyad_j\cT Q(\tau,\nabla u).
\end{align*} }
 Noticing that fact $\diver u=0$, we have
\begin{equation}\label{964456}
\ip{u\cdot\nabla u_j}{u_j}=0,
\qquad
\ip{u\cdot\nabla z_j}{z_j}=0.
\end{equation}
In the derivative of the cross term
$\ip{\Lambda u_j}{z_j}$, use
\begin{equation}\label{525622}
\Lambda(u\cdot\nabla u_j)
=u\cdot\nabla\Lambda u_j+[\Lambda,u\cdot\nabla]u_j.
\end{equation}

The first term in \eqref{525622} cancels with
$\ip{\Lambda u_j}{u\cdot\nabla z_j}$ after integration by parts, since $\diver u=0$ and $z_j$ is divergence-free (so that the Leray projector in the $u$-equation may be dropped in the scalar product with $z_j$).  Hence only
$[\Lambda,u\cdot\nabla]u_j$ remains, and the third estimate of Lemma \ref{lem:actives-com} gives
\begin{equation*}
\sum_{j\le j_0}2^{j\scal}
\| [\Lambda,u\cdot\nabla]u_j\|_2
\le C\|\nabla u\|_{\dot B^{d/2}_{2,1}}
\|u^\ell\|_{\dot B^{\scal}_{2,1}}
\le C\mathcal A(t)\mathcal D(t).
\end{equation*}
  Applying Lemma
\ref{lem:chen-hao} to the stress equation gives  
\begin{equation}\label{525235}
\partial_tq+\mathbb P(u\cdot\nabla q)+\beta_2\Lambda^2u=R_q,
\end{equation}
where
\begin{equation*} 
R_q=-\mathbb P(\nabla u\cdot\nabla\tau)
+\mathbb P\bigl(\nabla u\cdot\nabla\Delta^{-1}\diver\diver\tau\bigr)
-\mathbb P\diver Q(\tau,\nabla u).
\end{equation*}
Thus, for $j\ge j_0-1$, the forcing terms in
\eqref{52466} are
\begin{align*}
\mathcal F_j^{\rm rem}
&:=-\mathbb P[\dyad_j,u\cdot\nabla]u,
\\
\mathcal G_j^{\rm rem}
&:=-\mathbb P[\dyad_j,u\cdot\nabla]q+\dyad_jR_q.
\end{align*}

	As above, the Leray projectors disappear in scalar products because $u_j,q_j$ are
divergence-free.  The principal high-frequency transports satisfy
\begin{align*}
	\ip{u\cdot\nabla u_j}{u_j}&=0,
	&\ip{u\cdot\nabla q_j}{q_j}&=0, \\
	\ip{u\cdot\nabla u_j}{q_j}
	+\ip{u_j}{u\cdot\nabla q_j}&=0,	
\end{align*}
which imply that the principal transport also cancels in the mixed high-frequency energy \eqref{5372164}.  

\smallskip
\noindent\emph{Step 4: Summation of nonlinear remainders.}
The standard critical commutator estimate gives an $\ell^1$ sequence $(c_j)$,
$\sum_jc_j\le1$, such that
\begin{equation}\label{52523533}
\norm{[\dyad_j,u\cdot\nabla]f}_2
\le Cc_j2^{-j(d/2-1)}
\norm{\nabla u}_{\Bdot^{d/2}_{2,1}}
\norm f_{\Bdot^{d/2-1}_{2,1}}.
\end{equation}
The high-frequency norms are interpreted in the hybrid sense.  In particular,
\begin{equation}\label{525235331}
\norm{\nabla u}_{\Bdot^{d/2}_{2,1}}
\le C\cD(t),
\qquad
\norm{q}_{\Bdot^{d/2-1}_{2,1}}
+\norm{\tau}_{\Bdot^{d/2-1}_{2,1}}
\le C\cA(t).
\end{equation}
For the non-transport terms, use the divergence forms
\begin{align*}
[\nabla u\cdot\nabla\tau]_i
&=\partial_k(\partial_ju_k\,\tau_{ij}),
 \\
[\nabla u\cdot\nabla\Delta^{-1}\diver\diver\tau]_i
&=\partial_k\bigl(\partial_i u_k\,
\Delta^{-1}\diver\diver\tau\bigr),
\end{align*}
which follow from $\diver u=0$.   

Using Bony's decomposition, Bernstein's
inequality, and the critical hybrid product estimates in
Lemma \ref{lem:critical-h-m}, Corollary \ref{cor:critical-h-O}, Lemma \ref{lem:actives-com} gives,
\begin{align}
\norm{([\cT,u\cdot\nabla]\tau)^\low}_{\Bdot^{d/2-1}_{2,1}}
+\norm{(\cT Q(\tau,\nabla u))^\low}_{\Bdot^{d/2-1}_{2,1}}
&\le C\cA(t)\cD(t),
\label{963}\\
\norm{R_q^\high}_{\Bdot^{d/2-1}_{2,1}}
&\le C\cA(t)\cD(t),
\label{9633}
\end{align}
which together with Lemma \ref{lem:m-product} and Lemma \ref{lem:tran-c-1} yields
\begin{align}
&\sum_{j\le j_0}2^{j(d/2-1)}
(\norm{F_j^{\rm rem}}_2+\norm{G_j^{\rm rem}}_2)
\notag\\
&\quad+
\sum_{j\ge j_0-1}2^{j(d/2-1)}
(\norm{\mathcal F_j^{\rm rem}}_2+\norm{\mathcal G_j^{\rm rem}}_2)
\le C\cA(t)\cD(t).
\label{63634}
\end{align}
For completeness, the individual product estimates underlying
\eqref{63634} are
\begin{equation}
\norm{u\cdot\nabla u}_{\Bdot^{d/2-1}_{2,1}}
\le C\norm u_{\Bdot^{d/2-1}_{2,1}}
\norm u_{\Bdot^{d/2+1}_{2,1}},
\label{636634}
\end{equation}
which follows from Lemma~\ref{lem:m-product} applied to $\diver(u\otimes u)$, together with the hybrid estimates \eqref{eq:hyb-Q-p-app} and \eqref{eq:hyb-p-app} of Corollary~\ref{cor:critical-h-O}, in which the low-frequency source is measured in the $z$-equation (one derivative being gained through $\Lambda^{-1}$) and the high-frequency source in the $q$-equation. Here $\dot B^{d/2-1,d/2}_{2,1}$ denotes the hybrid stress space \eqref{eq:h-two-in}, and $\norm{\tau}_{\dot B^{d/2-1,d/2}_{2,1}}\le C\cA(t)$, $\norm{u}_{\Bdot^{d/2+1}_{2,1}}\le C\cD(t)$. Together with \eqref{52523533}, these estimates show that the total nonlinear contribution is bounded by $C\cA(t)\cD(t)$.

Finally, set
\begin{equation*} 
\mathscr E_a(t)=
\sum_{j\le j_0}2^{j(d/2-1)}L_j(t)
+C\sum_{j_0<j<j_h}2^{j(d/2-1)}{L_j(t)}
+\sum_{j\ge j_h}2^{j(d/2-1)}H_j(t).
\end{equation*}
\begin{align*} 
	\mathscr D_a(t)=
	 \sum_{j\le j_0}2^{j(\scal+2)}L_j(t)
	+C\sum_{j_0<j<j_h}2^{j(d/2-1)}{L_j(t)}
	+\sum_{j\ge j_h}2^{j(d/2-1)}(2^{2j}\norm{u_j}_2+\norm{q_j}_2).
\end{align*}
 The equivalences \eqref{523239} and \eqref{8544541} give
\eqref{eq:Lyap-equi}.  Summing
\eqref{53764} and \eqref{854454}, then using
\eqref{63634}, proves
\eqref{52523}.  Since
$\sup_t\cA(t)\le C\cX_0$ by \eqref{6364553}, choosing
$\cX_0\le c_0/(2C)$ yields \eqref{525232}.  Thus, the Proposition \ref{prop:act-Lyap} is completed.

\end{proof}

\begin{corollary}\label{cor:high-nonlinear}
Under the assumptions of \cref{thm:known-global}, for any $T>0$,
\begin{align*}
&\norm{u^\high}_{\CL^\infty_T(\Bdot^{d/2-1}_{2,1})}
+\norm{z^\high}_{\CL^\infty_T(\Bdot^{d/2}_{2,1})}
+\norm{u^\high}_{L^1_T(\Bdot^{d/2+1}_{2,1})}
+\norm{z^\high}_{L^1_T(\Bdot^{d/2}_{2,1})}
\notag\\
& \le C\left(
\norm{u_0^\high}_{\Bdot^{d/2-1}_{2,1}}
+\norm{z_0^\high}_{\Bdot^{d/2}_{2,1}}
+\int_0^T\cA(t)\cD(t)\,\dd t
\right)
\le C\cX_0.
\end{align*}
Equivalently, with $q=\Lambda z=\mathbb P\operatorname{div}\tau$,
\begin{align*}
&\norm{u^\high}_{\CL^\infty_T(\Bdot^{d/2-1}_{2,1})}
+\norm{q^\high}_{\CL^\infty_T(\Bdot^{d/2-1}_{2,1})}
+\norm{u^\high}_{L^1_T(\Bdot^{d/2+1}_{2,1})}
+\norm{q^\high}_{L^1_T(\Bdot^{d/2-1}_{2,1})}
\le C\cX_0.
\end{align*}

\end{corollary} 



\section{{Upper decay for nonlinear system}} \label{sec:non-a}
\subsection{The exact nonlinear active system}

It is more convenient to rewrite the nonlinear system \eqref{eq:nonlinear-1} as
\begin{equation}\label{5223}
\partial_tU=\cL U+\cN(u,\tau),
\end{equation}
where $U=(u,z)$, $\cL$ is the operator in \eqref{64238682} and
\begin{equation}\label{42222}
\cN(u,\tau)=
\begin{pmatrix}
-\Pj(u\cdot\nabla u)\\[2mm]
-\cT(u\cdot\nabla\tau)-\cT Q(\tau,\nabla u)
\end{pmatrix}.
\end{equation}



\subsection{Estimates for nonlinear terms}

Recall the instantaneous critical amplitude $\cA(t)$ from
\eqref{eq:criti-11}.  The global estimate \eqref{6364553} implies
\begin{equation*} 
		\tilde{\cA}_T:=\norm{(u,\tau)^\low}_{\CL^\infty_T(\Bdot^{d/2-1}_{2,1})}
		+\norm{u^\high}_{\CL^\infty_T(\Bdot^{d/2-1}_{2,1})}
		+\norm{\tau^\high}_{\CL^\infty_T(\Bdot^{d/2}_{2,1})}\le C\cX_0
		\qquad\text{for every }T>0.
\end{equation*}
For $-d/2<s<d/2-1$, set
\begin{equation}\label{eq:hybrid-Ds}
	\cD_s(t):=
	\norm{U^\low(t)}_{\Bdot^{s+2}_{2,\infty}}
	+\norm{u^\high(t)}_{\Bdot^{d/2+1}_{2,1}}
	+\norm{z^\high(t)}_{\Bdot^{d/2}_{2,1}}.
\end{equation}

\begin{proposition}\label{636543}
	Assume $d\ge2$ and let $(u,\tau)$ be the global solution established in Theorem \ref{thm:known-global}.
	Let $-d/2<s<d/2-1$. The nonlinear term  \eqref{42222} satisfies
	\begin{equation}\label{eq:hy-low-s}
		\norm{\cN(u,\tau)^\low}_{\Bdot^s_{2,\infty}}
		\le C\cA(t)\cD_s(t).
	\end{equation}
	{Moreover}, for any $T>0$,
	\begin{align}
		\norm{\cN(u,\tau)^\low}_{\CL^1_T(\Bdot^s_{2,\infty})}
		\le{}&C\tilde{\cA}_T
		\Bigl(
		\norm{U^\low}_{\CL^1_T(\Bdot^{s+2}_{2,\infty})}	+\norm{u^\high}_{L^1_T(\Bdot^{d/2+1}_{2,1})}
		+\norm{z^\high}_{L^1_T(\Bdot^{d/2}_{2,1})}
		\Bigr).
		\label{eq:hy-low-time}
	\end{align}
\end{proposition}

\begin{proof}
	We use the decomposition $\cN=(\cN_1,\cN_2)^T$ with
	\begin{equation*} 
		\cN_1=-\Pj(u\cdot\nabla u),
		\qquad
		\cN_2=-u\cdot\nabla z-[\cT,u\cdot\nabla]\tau-\cT Q(\tau,\nabla u).
	\end{equation*}
	For the velocity nonlinearity, using Lemma \ref{lem:bal-p-1} and the high-frequency embedding 
	\begin{equation*} 
		\norm{u^\high}_{\Bdot^{s+2}_{2,\infty}}
		\le C\norm{u^\high}_{\Bdot^{d/2+1}_{2,1}}
	\end{equation*}
	imply that
	\begin{align}
		\norm{\cN_1^\low}_{\Bdot^s_{2,\infty}}
		&\le C\norm u_{\Bdot^{d/2-1}_{2,1}}
		\norm u_{\Bdot^{s+2}_{2,\infty}}
		\le C\cA(t)\cD_s(t).
		\label{eq:N1-low-bound}
	\end{align}
	Here and below the full critical norms are in the low/high sense of \cref{thm:known-global}.
	
We next estimate $u\cdot\nabla z$. Consider the split
\[
u\cdot\nabla z
=u\cdot\nabla z^\low+u\cdot\nabla z^\high.
\]
By Lemma~\ref{lem:bal-p-1},
\begin{equation}\label{eq:uz-lowz}
\norm{(u\cdot\nabla z^\low)^\low}_{\Bdot^s_{2,\infty}}
\le C\norm{u}_{\Bdot^{s_c}_{2,1}}
       \norm{z^\low}_{\Bdot^{s+2}_{2,\infty}}.
\end{equation}

For the second term, Bony's decomposition gives, componentwise,
\[
u\cdot\nabla z^\high
=\dot T_u\nabla z^\high
 +\dot T_{\nabla z^\high}u
 +\dot R(u,\nabla z^\high).
\]
After low-frequency projection, only finitely many dyadic
blocks near $j_0$ contribute to either paraproduct.
Using Bernstein's inequality and
\[
\norm{\dot S_{k-1}u}_{L^\infty}
\le C2^k\norm{u}_{\Bdot^{s_c}_{2,1}},
\qquad
\norm{\dot S_{k-1}\nabla z^\high}_{L^\infty}
\le C2^k\norm{z^\high}_{\Bdot^{d/2}_{2,1}},
\]
we obtain
\begin{equation}\label{634634657}
\begin{aligned}
&\norm{(\dot T_u\nabla z^\high)^\low}_{\Bdot^s_{2,\infty}}
+\norm{(\dot T_{\nabla z^\high}u)^\low}_{\Bdot^s_{2,\infty}} \le
C\norm{u}_{\Bdot^{s_c}_{2,1}}
 \norm{z^\high}_{\Bdot^{d/2}_{2,1}}.
\end{aligned}
\end{equation}

For the remainder, the Fourier support properties and
Bernstein's inequality imply, for $j\le j_0$,
\[
\begin{aligned}
2^{js}\norm{\dyad_j\dot R(u,\nabla z^\high)}_{L^2}
&\le C2^{j(s+d/2)}
\sum_{k\ge\max\{j-N_0,j_0-N_0\}}
2^k\norm{\dyad_k u}_{L^2}
    \norm{\widetilde{\dyad}_k z^\high}_{L^2},
\end{aligned}
\]
where $N_0$ is a fixed integer determined by the dyadic
decomposition. Set
\[
a_k:=2^{ks_c}\norm{\dyad_k u}_{L^2},
\qquad
b_k:=2^{kd/2}
     \norm{\widetilde{\dyad}_k z^\high}_{L^2}.
\]
Since $d\ge2$ and $k\ge j_0-N_0$,
\[
2^k\norm{\dyad_k u}_{L^2}
    \norm{\widetilde{\dyad}_k z^\high}_{L^2}
=2^{k(2-d)}a_kb_k
\le Ca_kb_k.
\]
Therefore,
\[
2^{js}\norm{\dyad_j\dot R(u,\nabla z^\high)}_{L^2}
\le C2^{j(s+d/2)}
\norm{u}_{\Bdot^{s_c}_{2,1}}
\norm{z^\high}_{\Bdot^{d/2}_{2,1}}.
\]
Taking the supremum over $j\le j_0$ and using $s+d/2>0$,
then combining this estimate with
\eqref{634634657}, yields
\begin{equation}\label{eq:uz-high-bound}
\norm{(u\cdot\nabla z^\high)^\low}_{\Bdot^s_{2,\infty}}
\le C\norm{u}_{\Bdot^{s_c}_{2,1}}
       \norm{z^\high}_{\Bdot^{d/2}_{2,1}}.
\end{equation}
Combining \eqref{eq:uz-lowz} and \eqref{eq:uz-high-bound},
we conclude that
\begin{equation}\label{eq:uz-final}
\begin{aligned}
\norm{(u\cdot\nabla z)^\low}_{\Bdot^s_{2,\infty}}
&\le C\norm{u}_{\Bdot^{s_c}_{2,1}}
\left(
\norm{z^\low}_{\Bdot^{s+2}_{2,\infty}}
+\norm{z^\high}_{\Bdot^{d/2}_{2,1}}
\right)\\
&\le C\cA(t)\cD_s(t).
\end{aligned}
\end{equation}

	For the commutator, applying Lemma \ref{lem:zero-o-c} gives
	\begin{align}
		\norm{[\cT,u\cdot\nabla]\tau}_{\Bdot^s_{2,\infty}}
		&\le C\norm\tau_{\Bdot^{d/2-1}_{2,1}}
		\norm u_{\Bdot^{s+2}_{2,\infty}}
		\le C\cA(t)\cD_s(t),
		\label{eq:Tcomm-final}
	\end{align}
	here we used
	\begin{equation*} 
		\norm\tau_{\Bdot^{d/2-1}_{2,1}}
		\le C\left(
		\norm{\tau^\low}_{\Bdot^{d/2-1}_{2,1}}
		+\norm{\tau^\high}_{\Bdot^{d/2}_{2,1}}
		\right)
		\le C\cA(t).
	\end{equation*}
	Finally, using \cref{lem:bal-p-1} implies
	\begin{equation}\label{eq:Q-final}
		\norm{\cT Q(\tau,\nabla u)}_{\Bdot^s_{2,\infty}}
		\le C\norm\tau_{\Bdot^{d/2-1}_{2,1}}
		\norm u_{\Bdot^{s+2}_{2,\infty}}
		\le C\cA(t)\cD_s(t).
	\end{equation}
Combining \eqref{eq:N1-low-bound}, \eqref{eq:uz-final},
\eqref{eq:Tcomm-final}, and \eqref{eq:Q-final}
proves \eqref{eq:hy-low-s}.

To prove \eqref{eq:hy-low-time}, we apply H\"older's inequality
in time to each dyadic interaction before summing over
frequencies. The critical factor is measured in $L^\infty_T$
and the dissipative factor in $L^1_T$. For example,
\[
\begin{aligned}
&\norm{\dyad_k u\cdot
\nabla\widetilde{\dyad}_k z^\high}_{L^1_T(L^1)}\le
C2^k\norm{\dyad_k u}_{L^\infty_T(L^2)}
     \norm{\widetilde{\dyad}_k z^\high}_{L^1_T(L^2)}.
\end{aligned}
\]
The same dyadic summations as above, together with the
corresponding time-dependent product and commutator
estimates, give
\[
\begin{aligned}
\norm{\cN(u,\tau)^\low}_{\CL^1_T(\Bdot^s_{2,\infty})}
\le{}&
C\left(
\norm{u}_{\CL^\infty_T(\Bdot^{s_c}_{2,1})}
+\norm{\tau}_{\CL^\infty_T(\Bdot^{s_c}_{2,1})}
\right)
\norm{u}_{\CL^1_T(\Bdot^{s+2}_{2,\infty})}\\
&+
C\norm{u}_{\CL^\infty_T(\Bdot^{s_c}_{2,1})}
\left(
\norm{z^\low}_{\CL^1_T(\Bdot^{s+2}_{2,\infty})}
+\norm{z^\high}_{L^1_T(\Bdot^{d/2}_{2,1})}
\right).
\end{aligned}
\]
Since the critical norms are bounded by $C\tilde{\cA}_T$, we have
\[
\norm{u}_{\CL^1_T(\Bdot^{s+2}_{2,\infty})}
\le C\left(
\norm{u^\low}_{\CL^1_T(\Bdot^{s+2}_{2,\infty})}
+\norm{u^\high}_{L^1_T(\Bdot^{d/2+1}_{2,1})}
\right).
\]
This proves \eqref{eq:hy-low-time}, with a constant
independent of $T$.
\end{proof}

\subsection{Uniform propagation}
We now give the propagation of active negative regularity.
\begin{theorem}\label{thm:negative-p}
Assume $d\ge2$ and let $(u,\tau)$ be the global solution established in Theorem \ref{thm:known-global}. Recall that $U=(u,z)$ with $z=\cT\tau$. Assume
\begin{equation*} 
-\frac d2<s<\frac d2-1,
\qquad
U_0^\low\in\Bdot^{s}_{2,\infty}.
\end{equation*}
Then for any $T>0$,
\begin{align}
&\norm{U^\low}_{\CL^\infty_T(\Bdot^s_{2,\infty})}
+\norm{U^\low}_{\CL^1_T(\Bdot^{s+2}_{2,\infty})}\le C\left(
\norm{U_0^\low}_{\Bdot^s_{2,\infty}}+\cX_0
\right),
\label{eq:neg-prop}
\end{align}
where the constant $C$ is independent of $T$.
\end{theorem}

\begin{proof}
Applying the forced low-frequency estimate \eqref{5625222} to \eqref{5223} to give
\begin{align}
&\norm{U^\low}_{\CL^\infty_T(\Bdot^s_{2,\infty})}
+\norm{U^\low}_{\CL^1_T(\Bdot^{s+2}_{2,\infty})}\le C\norm{U_0^\low}_{\Bdot^s_{2,\infty}}
+C\norm{\cN(u,\tau)^\low}_{\CL^1_T(\Bdot^s_{2,\infty})}.
\label{eq:neg-source}
\end{align}
By \eqref{eq:hy-low-time},
\begin{align}
\norm{\cN^\low}_{\CL^1_T(\Bdot^s_{2,\infty})}
\le{}&C\cX_0
\norm{U^\low}_{\CL^1_T(\Bdot^{s+2}_{2,\infty})}+C\cX_0\left(
\norm{u^\high}_{L^1_T(\Bdot^{d/2+1}_{2,1})}
+\norm{z^\high}_{L^1_T(\Bdot^{d/2}_{2,1})}
\right).
\label{eq:sou-t-insert}
\end{align}

The last term is bounded by $C\cX_0$ through \eqref{6364553}. Choose $\cX_0$ sufficiently small so that the first term on the right-hand side of \eqref{eq:sou-t-insert} is absorbed by the left-hand side of \eqref{eq:neg-source}.\footnote{As usual, the absorption is formal as long as the finiteness of $\norm{U^\low}_{\CL^1_T(\Bdot^{s+2}_{2,\infty})}$ is not known a priori. It is rigorously justified by carrying out the argument on the Friedrichs approximations of \eqref{eq:GOB} (for which all low-frequency norms are finite on every bounded time interval, the data having finitely many dyadic blocks) and passing to the limit using the uniform bounds; we omit these routine details.} This yields
\[
\norm{U^\low}_{\CL^\infty_T(\Bdot^s_{2,\infty})}
+\norm{U^\low}_{\CL^1_T(\Bdot^{s+2}_{2,\infty})}
\le C\norm{U_0^\low}_{\Bdot^s_{2,\infty}}+C\cX_0^2.
\]
Since $\cX_0\le\varepsilon_0<1$, the term $\cX_0^2$ is bounded by $\cX_0$, proving \eqref{eq:neg-prop} uniformly in $T$.
\end{proof}

\subsection{Upper-bound decay estimates}

Inspired by the work of Xin and Xu \cite{Xin-Xu-2021}, we now establish the optimal upper decay of the nonlinear active pair.
\begin{theorem} \label{thm:nonlinear-upper}
Under the assumptions of \cref{thm:negative-p}, for any
\begin{equation*} 
s<\sigma\le\frac d2-1,
\end{equation*}
one has
\begin{equation}\label{eq:non-upper}
\norm{U(t)}_{\cH^{\sigma}}
\le C\braket t^{-\frac12(\sigma-s)}.
\end{equation}
In particular, at the critical active level,
\begin{equation}\label{eq:cri-decay}
\cE(t)
\le C\braket t^{-\frac12(d/2-1-s)}.
\end{equation}
The high-frequency terms appearing in the hybrid norm satisfy explicitly
\begin{equation}\label{eq:non-high-d}
\norm{u^\high(t)}_{\Bdot^{d/2-1}_{2,1}}
+\norm{z^\high(t)}_{\Bdot^{d/2}_{2,1}}
\le C\braket t^{-\frac12(d/2-1-s)}.
\end{equation}
\end{theorem}
 
\begin{proof}
 
Set
\begin{equation*}\label{eq:alpha0}
\alpha_0:=\frac12\left(\frac d2-1-s\right)>0.
\end{equation*}
By \cref{thm:negative-p},
\begin{equation*}
M_s:=1+\sup_{t\ge0}\norm{U^\low(t)}_{\Bdot^s_{2,\infty}}<\infty.
\end{equation*}

Interpolating the low-frequency norm between the indices $s$ and
$\scal+2$ at the target index $\scal$ implies

\begin{equation}\label{63344}
\norm{U^\low}_{\Bdot^{d/2-1}_{2,1}}
\le C M_s^{\theta_0}
\norm{U^\low}_{\Bdot^{d/2+1}_{2,1}}^{1-\theta_0},
\end{equation}
with the interpolation exponent 
\begin{equation*} 
	\theta_0=\frac{2}{d/2+1-s},
	\qquad
	1-\theta_0=\frac{d/2-1-s}{d/2+1-s}.
\end{equation*}
We then get
\begin{equation}\label{6334432}
\norm{U^\low}_{\Bdot^{d/2+1}_{2,1}}
\ge C M_s^{-\theta_0/(1-\theta_0)}
\norm{U^\low}_{\Bdot^{d/2-1}_{2,1}}^{1/(1-\theta_0)}.
\end{equation}
Since
\begin{equation*} 
q:=\frac1{1-\theta_0}
=1+\frac{2}{d/2-1-s}>1,
\end{equation*}
we have
\begin{equation}\label{eq:Dlow-power}
{\cD^\low(t):=\norm{U^\low}_{\Bdot^{d/2+1}_{2,1}}\ge cM_s^{-(q-1)}\bigg(\norm{U^\low}_{\Bdot^{d/2-1}_{2,1}}\bigg)^q:= cM_s^{-(q-1)}\cE_\low(t)^q.}
\end{equation}

For high-frequency,  using Bernstein's inequality implies
\begin{equation}\label{eq:Dhigh-Ehigh}
\begin{aligned}
\cD^\high(t)
&=\norm{u^\high}_{\Bdot^{d/2+1}_{2,1}}
+\norm{z^\high}_{\Bdot^{d/2}_{2,1}}\\
&\ge c\norm{u^\high}_{\Bdot^{d/2-1}_{2,1}}
+c\norm{z^\high}_{\Bdot^{d/2}_{2,1}}
=c\cE_\high(t).
\end{aligned}
\end{equation} 
Recalling that $\cE(t)\le C\varepsilon_0<1$, then
\begin{equation}\label{eq:Dhigh-power}
\cD^\high(t)\ge cM_s^{-(q-1)}{\cE_\high(t)}^q.
\end{equation}
Using $(x+y)^q\le2^{q-1}(x^q+y^q)$, \eqref{eq:Dlow-power} and \eqref{eq:Dhigh-power} imply
\begin{equation}\label{eq:D-E-power}
\cD(t)\ge cM_s^{-(q-1)}\cE(t)^q.
\end{equation}
By \eqref{eq:Lyap-equi}, the preceding coercivity inequality implies
\begin{equation*} 
\mathscr D_a(t)\ge cM_s^{-(q-1)}\mathscr E_a(t)^q.
\end{equation*}
Insert this into \eqref{525232} implies
\begin{equation}\label{eq:E-ode}
\frac{\dd}{\dd t}\mathscr E_a(t)
+cM_s^{-(q-1)}\mathscr E_a(t)^q\le0.
\end{equation}
Integrate the
differential inequality on intervals where $\mathscr E_a>0$ and use
\eqref{eq:Lyap-equi} give
\begin{align*}
\mathscr E_a(t)^{1-q}
&\ge \mathscr E_a(0)^{1-q}
+c(q-1)M_s^{-(q-1)}t,  
\end{align*}
and
\begin{align*}
\cE(t)
&\le C\mathscr E_a(t)
\le C(1+t)^{-1/(q-1)}
=C(1+t)^{-\alpha_0}.
\end{align*}
   Since
\[
\norm{u^\high}_{\Bdot^{d/2-1}_{2,1}}
+\norm{z^\high}_{\Bdot^{d/2}_{2,1}}
\le \cE(t),
\]
\eqref{eq:non-high-d} follows immediately.  

Recalling that $s<\sigma\le d/2-1$ and then using interpolation inequality gives   
\begin{equation*} 
\norm{U^\low}_{\Bdot^\sigma_{2,1}}
\le C
\norm{U^\low}_{\Bdot^s_{2,\infty}}^{\theta}
\norm{U^\low}_{\Bdot^{d/2-1}_{2,1}}^{1-\theta},
\qquad
\theta=\frac{d/2-1-\sigma}{d/2-1-s}.
\end{equation*}
Since
\[
(1-\theta)\alpha_0
=\frac{\sigma-s}{d/2-1-s}\cdot\frac{d/2-1-s}{2}
=\frac{\sigma-s}{2},
\]
we obtain
\begin{equation}\label{eq:low-g-decay}
\norm{U^\low(t)}_{\Bdot^\sigma_{2,1}}
\le C\braket t^{-\frac12(\sigma-s)}.
\end{equation}
The high-frequency part of $U$ is bounded by $\cE(t)$, whose decay exponent is at least $(\sigma-s)/2$. Adding the three components in \eqref{eq:Hrho} proves \eqref{eq:non-upper}.
\end{proof}

\begin{remark} 
The decay proof of \cite{ZhaiDecay2019} assumes a negative Besov norm of the whole pair $(u_0,\tau_0)$. The argument above shows that the active negative norm of $(u_0,\Lamb^{-1}\Pj\diver\tau_0)$ is sufficient for active decay. The passive stress is needed only in its small critical norm, where it acts as a coefficient in \eqref{eq:hy-low-s}. This formulation is compatible with the stationary stresses in
\cref{thm:stationary-1}.
\end{remark}

\section{{Proof of the nonlinear decay characterization}}\label{sec:nonlinear-transfer}
Let $S(t)=e^{t\cL}$ denote the semigroup of the active linear system \eqref{eq:l-1} and define

\begin{equation}
U_L(t):=S(t)U_0,
\qquad
R(t):=U(t)-U_L(t).
\label{eq:ULR}
\end{equation}
The exact Duhamel equation is
\begin{equation}
R(t)=\int_0^t S(t-s)\cN(u,\tau)(s)\,\dd s,
\label{eq:DuhamelR}
\end{equation}
where 
\begin{align*}
\cN(u,\tau)
&={}
\begin{pmatrix}
-\Pj(u\cdot\nabla u)\\[1mm]
-\cT(u\cdot\nabla\tau)-\cT Q(\tau,\nabla u)
\end{pmatrix}
\notag\\
&={}
\begin{pmatrix}
-\Pj(u\cdot\nabla u)\\[1mm]
-u\cdot\nabla z-[\cT,u\cdot\nabla]\tau-\cT Q(\tau,\nabla u)
\end{pmatrix}.
\end{align*}

Recall the critical amplitude $\cA(t)$ and the active dissipation $\cD(t)$ from \eqref{eq:criti-11} and \eqref{eq:active-d-11}. For the global small solution of \cref{thm:known-global,prop:act-Lyap},
\begin{equation}\label{eq:AD-global}
\sup_{t\ge0}\cA(t)\le C \cX_0\ll1,
\qquad
\int_0^\infty \cD(t)\,\dd t\le C \cX_0.
\end{equation}

\begin{definition} \label{def:low-c0}
For $s\in\R$, define
\begin{equation}
\Besov{s }{2}{c_0}
:=\left\{
f \in\Besov{s}{2}{\infty}:
2^{js}\norm{\dot\Delta_j f}_{2}\longrightarrow0
\text{ as }j\to-\infty
\right\}.
\label{eq:c0-space}
\end{equation}
We write
\[
F^{\low}\in\CL^1\bigl(\Besov{s}{2}{c_0}\bigr)
\]
when the sequence $\bigl(2^{js}\|\dot\Delta_jF\|_{L^1(0,\infty;L^2)}\bigr)_{j\le j_0}$
is bounded and
\begin{equation}
2^{js}\norm{\dot\Delta_jF}_{L^1(0,\infty;L^2)}
\longrightarrow0
\quad(j\to-\infty).
\label{eq:c0-time}
\end{equation}
\end{definition}

\subsection{{Recovering initial regularity from decay}}
{
\begin{lemma}
\label{lem:same-order}
Assume $d\ge2$ and recall the definition of $\alpha$ from \eqref{eq:alpha}, consider
\begin{equation*}
 -\frac d2<\sigma_1<\sigma\le \scal.
\end{equation*}
Let $(u,\tau)$ be the global solution established in Theorem \ref{thm:known-global}, and assume that
\begin{equation}
\norm{U(t)}_{\cH^\sigma}\le M\braket t^{-\alpha}
\qquad(t\ge0).
\label{eq:observ-decay}
\end{equation}
 We then have
\begin{equation}
U_0^{\low}\in\Besov{\sigma_1}{2}{\infty}
\label{eq:B-neg-rec}
\end{equation}
and
\begin{equation}
{
\norm{R^{\low}(t)}_{\Besov{\sigma}{2}{1}}
\le C\bigl(M+\cX_0^2\bigr)\braket t^{-\alpha}}
\qquad(t\ge0).
\label{eq:same-order}
\end{equation}

If \eqref{eq:observ-decay} is substituted  by
\[
\norm{U(t)}_{\cH^\sigma}\le M t^{-\alpha}
\qquad(t\ge t_0\ge1),
\]
 the conclusion \eqref{eq:B-neg-rec} still holds true.
\end{lemma}
\begin{proof}
Set
\begin{equation}
 \tilde s:=\sigma-\sigma_1>0,
\qquad \gamma_1:=\scal-\sigma_1,\qquad \gamma_2:=\sigma_1+\frac d2,
\label{eq:B-gammas}
\end{equation}
The admissible range \eqref{eq:main-assumptions} gives
\begin{equation}
\gamma_1\ge\tilde s,\qquad \gamma_2>0.
\label{eq:B-gap-con}
\end{equation}

\medskip
\noindent\textbf{Step 1: A key $\ell^1$ norm of critical sequence.}
For $j\le j_0$, define
\begin{equation*}
X_j:=2^{j\sigma_1}\norm{\dyad_jU}_{L^\infty_tL^2_x}
+2^{j(\sigma_1+2)}\norm{\dyad_jU}_{L^1_tL^2_x},
\qquad
Y_j:=2^{j\sigma_1}\norm{\dyad_j\cN(u,\tau)}_{L^1_tL^2_x}.
\label{eq:B-XY}
\end{equation*}
Since $z=\cT\tau$, the critical estimate
\eqref{6364553} gives 
\begin{equation}
\sum_{j\le j_0}2^{j\gamma_1}X_j
=\sum_{j\le j_0}\left(
2^{j\scal}\norm{\dyad_jU}_{L^\infty_tL^2_x}
+2^{j(\scal+2)}\norm{\dyad_jU}_{L^1_tL^2_x}
\right)
\le C \cX_0.
\label{6747543}
\end{equation}
Fix an integer $N_*$ large enough. 
For $k\le j_0+N_*$, set
\begin{equation}
c_k:=2^{k\scal}
\left(
\norm{\dyad_ku}_{L^\infty_tL^2_x}
+\norm{\dyad_k\tau}_{L^\infty_tL^2_x}
\right).
\label{eq:B-ck}
\end{equation}
Noticing that the high stress weight
$2^{k(\scal+1)}$ controls the weight in \eqref{eq:B-ck}, then the critical global
estimate implies
\begin{equation*}
\sum_{k\le j_0+N_*}c_k\le C \cX_0.
\label{eq:B-c-l1}
\end{equation*}

Set
\begin{equation}
\varepsilon_j:=C\left(
\sum_{|m-j|\le N_*}c_m
+\sum_{m\le j-N_*}2^{m-j}c_m
\right),\qquad j\le j_0.
\label{eq:B-epsilon}
\end{equation}
Changing the order of summation yields
\begin{equation}
\sum_{j\le j_0}\varepsilon_j\le C \cX_0.
\label{eq:B-epsilon-l1}
\end{equation}
\medskip
\noindent\textbf{Step 2: Frequency-resolved estimate of the  source term.}
  For a representative low-frequency term
$Q(\tau,\nabla u)$, the three Bony interactions
satisfy
\begin{align}
	2^{j\sigma_1}\norm{\dyad_j(\dot T_\tau \nabla u)}_{L^1_tL^2_x}
	&\le C X_j\sum_{m<j+C}2^{m-j}c_m,
	\label{eq:B-para-1}\\
	2^{j\sigma_1}\norm{\dyad_j(\dot T_{\nabla u}\tau)}_{L^1_tL^2_x}
	&\le C c_j\sum_{k<j}2^{-(j-k)(\scal-\sigma_1)}X_k,
	\label{eq:B-para-2}\\
	2^{j\sigma_1}\norm{\dyad_j\dot R(\tau,\nabla u)}_{L^1_tL^2_x}
	&\le C\sum_{k\ge j-C}2^{-(k-j)(\sigma_1+d/2)}c_kX_k.
	\label{eq:B-remainder}
\end{align}
where  \eqref{eq:B-para-1} comes from 
$\norm{\dot S_{j-1}\tau}_{L^\infty_x}
\lesssim\sum_{m<j}2^{md/2}\norm{\dyad_m\tau}_2$; \eqref{eq:B-para-2} uses
$\norm{\dot S_{j-1}\nabla u}_{L^1_tL^\infty_x}
\lesssim\sum_{k<j}2^{k(\scal-\sigma_1)}X_k$; and \eqref{eq:B-remainder} follows from the Bernstein inequality gain $2^{jd/2}$. 

The terms $\Pj(u\cdot\nabla u)$, $u\cdot\nabla z$ have similar form. For
$[\cT,u\cdot\nabla]\tau$, the localized mean-value formula used in \eqref{665453}  places the effective derivative on $u$ and gives the same two kernels. {Indeed,
\[
 \mathscr C(u,\tau)
 :=[\mathcal T,u\cdot\nabla]\tau,
 \qquad
 \mathcal T=\Lambda^{-1}\mathbb P\operatorname{div}.
\]

For the low-frequency, the definitions of
$c_k$ and $X_k$ give
\[
 \|\tau_k\|_{L^\infty_tL^2_x}\le 2^{-ks_c}c_k,
 \qquad
 \|u_k\|_{L^1_tL^2_x}\le 2^{-k(\sigma_1+2)}X_k.
\]
Consequently, Bernstein's inequalities imply
\[
 \|\nabla u_k\|_{L^1_tL^\infty_x}
 \lesssim 2^{k(s_c-\sigma_1)}X_k,
 \qquad
 \|\nabla\tau_k\|_{L^\infty_tL^\infty_x}
 \lesssim 2^{2k}c_k.
\]
Let $N_*$ be the (large) integer fixed in Step~1 and let $N_0$ denote a fixed integer depending only on the dyadic decomposition. We first consider interactions whose
two input indices are at most $j_0+N_*$ (low-frequency inputs); comparable high-frequency inputs are treated afterwards.
We split these interactions into low-high, high-low, and
comparable-frequency contributions, denoted by
$\mathscr C_j^{(1)}$, $\mathscr C_j^{(2)}$, and
$\mathscr C_j^{(3)}$, respectively.

\medskip
\noindent\textit{(i) Low-frequency velocity and high-frequency stress.}
Consider
\[
 \dot\Delta_j[\mathcal T,u_m\cdot\nabla]\tau_k,
 \qquad m\le k-N.
\]
The Fourier support properties imply that this expression vanishes
unless $|j-k|\le N_0$.
Choose a smooth annular cutoff $\chi$ and define
\[
 \mathcal T_k:=\mathcal T\chi(2^{-k}D).
\]
The convolution kernel $K_k$ of $\mathcal T_k$ satisfies
\[
 \int_{\mathbb R^d}|y|\,|K_k(y)|\,dy\lesssim 2^{-k}.
\]
 The commutation of   $\mathcal T$   with spatial derivatives yields 
\[
 [\mathcal T,u_m\cdot\nabla]\tau_k
 =
 \sum_{a=1}^d[\mathcal T_k,u_m^a]\partial_a\tau_k.
\]
Moreover,
\[
 \begin{aligned}
 [\mathcal T_k,u_m^a]\partial_a\tau_k(x)
 =\int_{\mathbb R^d}K_k(y)
 \bigl(u_m^a(x-y)-u_m^a(x)\bigr)
 \partial_a\tau_k(x-y)\,dy.
 \end{aligned}
\]
Using the fact
\[
 u_m^a(x-y)-u_m^a(x)
 =
 -\int_0^1y\cdot\nabla u_m^a(x-\theta y)\,d\theta
\]
and the first-moment bound for $K_k$, we obtain
\[
 \begin{aligned}
 \|[\mathcal T,u_m\cdot\nabla]\tau_k\|_{L^2_x}
 &\lesssim
 2^{-k}\|\nabla u_m\|_{L^\infty_x}
          \|\nabla\tau_k\|_{L^2_x}\\
 &\lesssim
 \|\nabla u_m\|_{L^\infty_x}\|\tau_k\|_{L^2_x}.
 \end{aligned}
\]

Thus the factor $2^{-k}$ supplied by the commutator cancels the derivative on the high-frequency stress.
Integrating in time and using $|j-k|\le N_0$ give
\[
 \begin{aligned}
 &2^{j\sigma_1}
 \|\dot\Delta_j[\mathcal T,u_m\cdot\nabla]\tau_k
       \|_{L^1_tL^2_x}\\
 & \lesssim
 2^{j\sigma_1}\|\nabla u_m\|_{L^1_tL^\infty_x}
          \|\tau_k\|_{L^\infty_tL^2_x}\\
 & \lesssim
 2^{j\sigma_1+m(s_c-\sigma_1)-ks_c}X_mc_k\\
 & \lesssim
 2^{-(k-m)(s_c-\sigma_1)}c_kX_m.
 \end{aligned}
\]

\medskip
\noindent\textit{(ii) High-frequency velocity and low-frequency stress.}
For $m\le k-N_*$ and noticing 
 the  support belongs to  $|j-k|\le N_0$,
we then estimate the two terms of the commutator separately 
\[
 \begin{aligned}
 &\|\dot\Delta_j\mathcal T(u_k\cdot\nabla\tau_m)
       \|_{L^1_tL^2_x}
 +\|\dot\Delta_j(u_k\cdot\nabla\mathcal T\tau_m)
       \|_{L^1_tL^2_x}\\
 & \lesssim
 \|u_k\|_{L^1_tL^2_x}\,
 2^{m(1+d/2)}\|\tau_m\|_{L^\infty_tL^2_x}.
 \end{aligned}
\]
Here we used {Bernstein's inequality}
\[
 \|\nabla\mathcal T\tau_m\|_{L^\infty_x}
 \lesssim 2^{m(1+d/2)}\|\tau_m\|_{L^2_x}.
\]
It follows that
\[
 \begin{aligned}
 &2^{j\sigma_1}
 \|\dot\Delta_j[\mathcal T,u_k\cdot\nabla]\tau_m
       \|_{L^1_tL^2_x}\\
 & \lesssim
 2^{j\sigma_1-k(\sigma_1+2)+2m}X_kc_m
 \lesssim 2^{2(m-k)}X_kc_m.
 \end{aligned}
\]

\medskip
\noindent\textit{(iii) Comparable-frequency inputs.}
For $|k-m|<N_*$ and 
$k\ge j-N_0$.
The localized multiplier estimate
\[
 \|\dot\Delta_j\mathcal T h\|_{L^2_x}
 +\|\dot\Delta_jh\|_{L^2_x}
 \lesssim 2^{jd/2}\|h\|_{L^1_x}
\]
and H\"older's inequality give
\[
 \begin{aligned}
 &\|\dot\Delta_j[\mathcal T,u_k\cdot\nabla]\tau_m
       \|_{L^1_tL^2_x}\\
 & \lesssim
 2^{jd/2}2^m
 \|u_k\|_{L^1_tL^2_x}
 \|\tau_m\|_{L^\infty_tL^2_x}.
 \end{aligned}
\]

For the second term of the commutator, we use
$\|\nabla\mathcal T\tau_m\|_2\lesssim 2^m\|\tau_m\|_2$.
Since $|k-m|<N_*$ and $s_c+1=d/2$, the dyadic weights satisfy
\[
 \begin{aligned}
 &2^{j\sigma_1}
 \|\dot\Delta_j[\mathcal T,u_k\cdot\nabla]\tau_m
       \|_{L^1_tL^2_x}\\
 & \lesssim
 2^{j(\sigma_1+d/2)}2^{m-ms_c-k(\sigma_1+2)}X_kc_m\\
 & \lesssim
 2^{-(k-j)(\sigma_1+d/2)}X_kc_m.
 \end{aligned}
\]    }
     After enlarging $N_*$ by a fixed
amount,  for $j\le j_0$,
\[
\dyad_j(f_k\nabla g_m)=0
\quad\hbox{if}\quad
\max\{k,m\}\ge j_0+N_*,\quad |k-m|>N_*.
\]
{Since the preceding estimate for the remainder term only covers interactions between comparable low frequencies, it remains to consider the following case: $k, m$ are comparable high frequencies.}  Then, for $k\ge j_0-N_*$,
\[
a_k^u:=2^{k\scal}\norm{u_k}_{L^\infty_tL^2_x},\qquad
a_k^\tau:=2^{k(\scal+1)}\norm{\tau_k}_{L^\infty_tL^2_x},\qquad
b_k^u:=2^{k(\scal+2)}\norm{u_k}_{L^1_tL^2_x}.
\]
The finitely many transition blocks are easy to deal with, and
\eqref{6364553} gives
\[
\sum_{k\ge j_0-N_*}(a_k^u+a_k^\tau+b_k^u)\le C \cX_0.
\]

Let
\[
\mathcal I:=\bigl\{(k,m):\max\{k,m\}\ge j_0+N_*,\ |k-m|\le N_*\bigr\}.
\]
Since $\diver u_k=0$, the two transport terms may be written as
\[
\Pj(u_k\cdot\nabla u_m)=\Pj\diver(u_k\otimes u_m),
\qquad
\cT(u_k\cdot\nabla\tau_m)=\cT\diver(u_k\otimes\tau_m).
\]
Hence the derivative in these two terms acts at frequency $2^j$.
Using that $\Pj$ and $\cT$ have order zero, Bernstein's inequality
$\norm{\dyad_jh}_2\lesssim2^{jd/2}\norm h_1$,  H\"older inequality in  spatial and then in time, we obtain for $(k,m)\in\mathcal I$
\begin{align*}
	\norm{\dyad_j\Pj(u_k\cdot\nabla u_m)}_{L^1_tL^2_x}
	&\lesssim 2^{j(1+d/2)}
	\norm{u_k}_{L^\infty_tL^2_x}\norm{u_m}_{L^1_tL^2_x},\\
	\norm{\dyad_j\cT(u_k\cdot\nabla\tau_m)}_{L^1_tL^2_x}
	&\lesssim 2^{j(1+d/2)}
	\norm{u_k}_{L^1_tL^2_x}\norm{\tau_m}_{L^\infty_tL^2_x},\\
	\norm{\dyad_j\cT Q(\tau_k,\nabla u_m)}_{L^1_tL^2_x}
	&\lesssim 2^{jd/2}
	\norm{\tau_k}_{L^\infty_tL^2_x}
	2^m\norm{u_m}_{L^1_tL^2_x}.
\end{align*}
and
{applying the estimates from part~(iii) with these high-frequency
weights, we obtain
\[
 \begin{aligned}
 &2^{j\sigma_1}
 \|\dot\Delta_j[\mathcal T,u_k\cdot\nabla]\tau_m
       \|_{L^1_tL^2_x}\\
 & \lesssim
 2^{j(\sigma_1+d/2)}
 2^{m-m(s_c+1)-k(s_c+2)}
 b_k^u a_m^\tau\\
 & \lesssim
 2^{j(\sigma_1+d/2)}2^{-kd}b_k^u a_m^\tau,
 \qquad |k-m|\le N_*.
 \end{aligned}
\]}

Because $|k-m|\le N_*$ and $\scal=d/2-1$, the three products of dyadic
weights satisfy, respectively,
\begin{align*}
	2^{-k\scal}2^{-m(\scal+2)}&\lesssim2^{-kd},\\
	2^{-k(\scal+2)}2^{-m(\scal+1)}&\lesssim2^{-k(d+1)},\\
	2^{-k(\scal+1)}2^{-m(\scal+1)}&\lesssim2^{-kd}.
\end{align*}
All indices in $\mathcal I$ are bounded from below by a fixed number depending
only on $j_0$ and $N_*$.  

Denote by $\cN_{\mathcal I}$ the sum of the contributions of the pairs $(k,m)\in\mathcal I$ to $\cN(u,\tau)$. Summing the preceding estimates and using the $\ell^1$ bounds for $a^u,a^\tau,b^u$ therefore yields
\begin{align*}
	2^{j\sigma_1}\norm{\dyad_j\cN_{\mathcal I}}_{L^1_tL^2_x}
	&\lesssim 2^{j(\sigma_1+d/2)}(1+2^j)
	\left(\sum_k(a_k^u+a_k^\tau)\right)\left(\sum_m b_m^u\right)\notag\\
	&\le C2^{j(\sigma_1+d/2)}\cX_0^2
	\le C \cX_0^2,
	\qquad j\le j_0,
\end{align*}
where $1+2^j\le C_{j_0}$. 

 Since $\sigma_1+d/2>0$, the factor
$2^{j(\sigma_1+d/2)}$ is uniformly bounded for $j\le j_0$ and tends to zero as
$j\to-\infty$.  Consequently
\begin{align}\label{23663466}
{	Y_j\le C \cX_0^2
	+C\varepsilon_j\sum_{k<j}2^{-(j-k)\gamma_1}X_k
	+C\sum_{j\le k\le j_0}2^{-(k-j)\gamma_2}\varepsilon_kX_k.}
\end{align}

\medskip
\noindent\textbf{Step 3: Recovery of the initial negative norm.}
	Let
\[
a_j:=2^{j\sigma_1}\norm{\dyad_jU_0}_2.
\label{eq:B-aj}
\]
 Integrating the estimate \eqref{eq:block-Duhamel} on
$(0,\infty)$ gives
\begin{equation}	 \label{eq:B-X-from-aY}
	 X_j\le C(a_j+Y_j).
\end{equation}
At the intrinsic diffusion time $t_j:=2^{-2j}$, using the   estimate \eqref{eq:d-two-s} gives
\begin{equation}\label{eq:B-aj-at-tj}
	 a_j\le C2^{j\sigma_1}\norm{\dyad_jU_L(t_j)}_2
	 \le C2^{j\sigma_1}\norm{\dyad_jU(t_j)}_2+C Y_j.
\end{equation}
By \eqref{eq:observ-decay}, we have
 \begin{equation}\label{eq:B-ob-at-tj}
2^{j\sigma_1}\norm{\dyad_jU(t_j)}_2
\le 2^{-j\tilde s}\norm{U(t_j)}_{\cH^\sigma}
\le C M,
 \end{equation}
where we used that $2^{-j\tilde s}\braket{2^{-2j}}^{-\tilde s /2}$ is uniformly bounded for
$j\le j_0$. Hence $a_j\le C(M+Y_j)$, and \eqref{eq:B-X-from-aY} gives
$X_j\le C(M+Y_j)$.

 Combining this {with \eqref{23663466} } and employing 
\eqref{eq:B-v-ineq} with $x_j=X_j$, $e_j=\varepsilon_j$, and
$A=C(M+\cX_0^2)$, using  \eqref{eq:B-gap-con},  
\eqref{eq:B-epsilon-l1},   \cref{key-le-10}  implies
{\begin{align*}\label{eq:B-uniform-XYa}
\sup_{j\le j_0}X_j+\sup_{j\le j_0}Y_j
+\sup_{j\le j_0}a_j
\le C(M+\cX_0^2),
\end{align*}}
where we used \eqref{eq:B-epsilon-l1} and decreased $\varepsilon_0$ in \eqref{eq:X0} so that $C\cX_0\le e_*$. We stress that \cref{key-le-10} only requires $\gamma_1>0$, so that the argument covers the critical endpoint $\sigma=\scal$, where $\gamma_1=\tilde s$.
Moreover,
\[
\norm{U_0^{\low}}_{\Besov{\sigma}{2}{1}}
\le \frac{2^{j_0 \tilde s}}{1-2^{-\tilde s}}
\norm{U_0^{\low}}_{\Besov{\sigma_1}{2}{\infty}},
\qquad \tilde s=\sigma-\sigma_1>0,
\]
recovers $U_0^{\low}\in\Besov{\sigma}{2}{1}$. The proof of
\cref{thm:li-upp-}, together with the active high-frequency datum
controlled by \eqref{eq:observ-decay}  yields
\[
\norm{U_L(t)}_{\cH^\sigma}
\le C(M+\cX_0^2)\braket t^{-\alpha}.
\]
Combining this estimate with \eqref{eq:observ-decay} and $R=U-U_L$ proves
\eqref{eq:same-order}.

It remains to prove the eventual-time assertion. Assume
\[
\norm{U(t)}_{\cH^\sigma}\le M t^{-\alpha}
\qquad(t\ge t_0\ge1),
\]
and choose $J\le j_0$ so that $2^{-2J}\ge t_0$.

 For $j\le J$, put
$t_j=2^{-2j}$, then
\[
2^{j\sigma_1}\norm{\dyad_jU(t_j)}_2
\le 2^{-j\tilde s}M t_j^{-\tilde s/2}
\le M.
\]
Using the Duhamel formula and the estimates
 \eqref{eq:block-Duhamel},\eqref{eq:d-two-s} imply, for any $j\le J$,
\begin{equation}\label{eq:eventual-aX}
	\begin{aligned}
		 	a_j:=2^{j\sigma_1}\norm{\dyad_jU_0}_2
		 &\le C2^{j\sigma_1}\norm{\dyad_jU_L(t_j)}_2
		 \le C(M+Y_j),\\
		 X_j&\le C(a_j+Y_j)\le C(M+Y_j).
	\end{aligned}
\end{equation}
In the last sum {of \eqref{23663466},} split the indices at $J$. Let
\[
F_J:=C\sum_{J<k\le j_0}
2^{-(k-J)\gamma_2}\varepsilon_kX_k.
\]

For $j\le J$, their contribution is
$2^{-(J-j)\gamma_2}F_J\le F_J$. Moreover,
\eqref{6747543} and \eqref{eq:B-epsilon-l1} give
\[
F_J
\le C \cX_0^2 2^{-J\gamma_1}
\sum_{n\ge1}2^{-n(\gamma_1+\gamma_2)}
\le C2^{-J\gamma_1}\cX_0^2<\infty.
\]
Consequently, on the half-axis $j\le J$, \eqref{23663466} and
\eqref{eq:eventual-aX} yield
\[
X_j\le A_J
+C_0\varepsilon_j\sum_{k<j}2^{-(j-k)\gamma_1}X_k
+C_0\sum_{j\le k\le J}2^{-(k-j)\gamma_2}\varepsilon_kX_k,
\]
where $A_J=C(M+\cX_0^2+F_J)$. The summability hypothesis \eqref{eq:B-v-hyp} of \cref{key-le-10} on the half-axis $j\le J$ follows from \eqref{6747543}, and, as before, $\sum_{j\le J}\varepsilon_j\le C\cX_0\le e_*$.

 {Applying the Lemma \ref{key-le-10}}
with upper endpoint $J$ therefore gives
\begin{equation}\label{eq:eventual-X}
	\sup_{j\le J}X_j\le C A_J<\infty.
\end{equation}
Substituting this  in \eqref{23663466} and following from
\eqref{eq:eventual-aX} gives
$\sup_{j\le J}(Y_j+a_j)<\infty$.

For the remaining finite family $J<j\le j_0$, the critical initial bound
\eqref{eq:X0} yields $a_j\le C2^{-J\gamma_1}\cX_0$. Hence
\begin{equation*}\label{eq:even-neg-re}
	 \norm{U_0^{\low}}_{\Besov{\sigma_1}{2}{\infty}}
	 =\sup_{j\le j_0}a_j
	 \le C\bigl(M+\cX_0^2+2^{-J\gamma_1}\cX_0\bigr)<\infty.
\end{equation*}
This proves the eventual-time assertion.  
\end{proof}
}

\subsection{The  nonlinear source with low-frequency \texorpdfstring{$c_0$}{c0} tail}
We next establish the low-frequency $c_0$ property for the nonlinear
source.

\begin{lemma}
\label{lem:c0-source}
Let
\begin{equation*}
-\frac d2<\sigma_1<\scal,
\qquad
U_0^{\low}\in\Besov{\sigma_1}{2}{\infty}.
\end{equation*}
Then
\begin{equation}
{
\cN(u,\tau)^{\low}
\in\CL^1\bigl(\Besov{\sigma_1}{2}{c_0}\bigr).}
\label{eq:c0-sou-r}
\end{equation}
Equivalently, if
\begin{equation}
b_j:=2^{j\sigma_1}
\norm{\dot\Delta_j\cN(u,\tau)}_{L^1(0,\infty;L^2)},
\qquad j\le j_0,
\label{eq:bj-def}
\end{equation}
then $(b_j)_{j\le j_0}$ is bounded and $b_j\to0$ as $j\to-\infty$.
\end{lemma}

\begin{proof}
	Define the norms of $X_c$ and $Y_{\sigma_1}$ by
	\[
	\norm{\tau}_{X_c}
	:=\norm{\tau^\ell}_{\CL^\infty(\Bdot^{s_c}_{2,1})}
	+\norm{\tau^h}_{\CL^\infty(\Bdot^{s_c+1}_{2,1})},
	\]
	\[
	\norm{u}_{Y_{\sigma_1}}
	:=\norm{u^\ell}_{\CL^1(\Bdot^{\sigma_1+2}_{2,\infty})}
	+\norm{u^h}_{L^1(\Bdot^{s_c+2}_{2,1})},
	\qquad
	\norm{z}_{Y_{\sigma_1}}
	:=\norm{z^\ell}_{\CL^1(\Bdot^{\sigma_1+2}_{2,\infty})}
	+\norm{z^h}_{L^1(\Bdot^{s_c+1}_{2,1})}.
	\]
Recalling from Theorem \ref{thm:negative-p} and Theorem \ref{thm:known-global} gives
\begin{align*}
U^{\low}
&\in
\CL^\infty\bigl(\Besov{\sigma_1}{2}{\infty}\bigr)
\cap
\CL^1\bigl(\Besov{\sigma_1+2}{2}{\infty}\bigr),
\\
(u,\tau)^{\low}
&\in\CL^\infty\bigl(\Besov{\scal}{2}{1}\bigr),
\qquad
u^{\high}\in\CL^\infty\bigl(\Besov{\scal}{2}{1}\bigr),
\qquad
\tau^{\high}\in\CL^\infty\bigl(\Besov{\scal+1}{2}{1}\bigr),
\end{align*}
with the corresponding active dissipative norms in $L^1_t$.

Let $\mathfrak B(f,g)$ denote any one of the following four bilinear expressions, where $f$ stands for the critical factor ($u$ or $\tau$) and $g$ for the dissipative factor ($u$ or $z$):
\begin{equation}
\Pj(u\cdot\nabla u),\quad
u\cdot\nabla z,\quad
[\cT,u\cdot\nabla]\tau,\quad
\cT Q(\tau,\nabla u).
\label{eq:source-C}
\end{equation}
Then, by \eqref{eq:hy-low-time} with $T=\infty$,
\begin{equation}
	\norm{Q(\tau,\nabla u)^{\low}}_{\CL^1(\Besov{\sigma_1}{2}{\infty})}
	\le C\norm{\tau}_{X_c}\norm{u}_{Y_{\sigma_1}}.
	\label{eq:bilinear-C}
\end{equation}
The remaining terms can be treated similarly.

Let
\begin{equation*}
f^{(N)}:=\sum_{|k|\le N}\dot\Delta_k f.
\end{equation*}
Noticing that the critical dyadic norm of $f$ is an $\ell^1$ norm,
\begin{equation}
\norm{f-f^{(N)}}_{X_c}\longrightarrow0
\quad(N\to\infty).
\label{eq:critical-C}
\end{equation}
This together \eqref{eq:bilinear-C} implies
\begin{equation*}
\mathfrak B(f^{(N)},g)^{\low}
\longrightarrow \mathfrak B(f,g)^{\low}
\quad\text{in }\CL^1\bigl(\Besov{\sigma_1}{2}{\infty}\bigr).
\end{equation*}

\noindent\textbf{Step 1:  The estimates of $\cT Q$.}
We first verify the estimate for the nonlinear term
\begin{equation*}
\mathfrak B_Q(\tau,u):=\cT Q(\tau,\nabla u).
\end{equation*}
Choose the critical $\ell^1$ factor to be the stress and set
\begin{equation*}
\tau^{(N)}:=\sum_{|k|\le N}\dot\Delta_k\tau.
\end{equation*}

Since each component of $Q(\tau,\nabla u)$ is a finite linear combination of
products $\tau_{ab}\,\partial_m u_n$, the support argument above yields, for
$j\le-N-J_*$,

\begin{equation*}
\dot\Delta_j\cT Q(\tau^{(N)},\nabla u)
=
\sum_{\substack{|k|\le N\\|\ell-k|\le C_0}}
\dot\Delta_j\cT Q(\dot\Delta_k\tau,
\nabla\dot\Delta_\ell u),
\end{equation*}
here we used the following support identity for a sufficiently
small $j\le-N-J_*$
\begin{equation*}
\dot\Delta_j\mathfrak B(f^{(N)},g)
=
\sum_{\substack{|k|\le N\\|\ell-k|\le C_0}}
\dot\Delta_j\mathfrak B(f_k,g_\ell).
\end{equation*}

Indeed, whenever \(|k-\ell|>C_0\), after possibly enlarging \(C_0\), the Fourier support of the product of a \(2^k\)-block and a \(2^\ell\)-block is contained in an annulus of size \(2^{\max\{k,\ell\}}\). Consequently,
\begin{equation*}
\dot\Delta_j\mathfrak B(f_k,g_\ell)=0
\quad\text{unless}\quad
|j-\max\{k,\ell\}|\le C_0.
\end{equation*}
   Bernstein's inequality gives
\begin{align}
&2^{j\sigma_1}
\norm{\dot\Delta_j\cT Q(\tau^{(N)},\nabla u)}_{L^1_tL^2_x}
\notag\\
& \le
C2^{j(\sigma_1+d/2)}
\sum_{\substack{|k|\le N\\|\ell-k|\le C_0}}
\norm{Q(\dot\Delta_k\tau,
\nabla\dot\Delta_\ell u)}_{L^1_tL^1_x}
\notag\\
&\le
C2^{j(\sigma_1+d/2)}
\sum_{\substack{|k|\le N\\|\ell-k|\le C_0}}
2^\ell
\norm{\dot\Delta_k\tau}_{L^\infty_tL^2_x}
\norm{\dot\Delta_\ell u}_{L^1_tL^2_x}.
\label{eq:Q-block-C}
\end{align}

For $\sigma_1<\scal$, using the fact $2^{k\scal}\le C_{j_0}2^{k(\scal+1)}$ for larger $k$ and
$2^{\ell(\sigma_1+2)}\le C_{j_0,\sigma_1}2^{\ell(\scal+2)}$ for large $\ell$, we have
\begin{align*}
\sum_{k\in\Z}a_k
 \le C\norm{\tau}_{ X_c},  \qquad
B_{\sigma_1}
 \le C\norm{u}_{  Y_{\sigma_1}}.
\end{align*}
with
\begin{equation*}
	 a_k:=2^{k\scal}
	 \norm{\dot\Delta_k\tau}_{L^\infty_tL^2_x}, \quad  B_{\sigma_1}:=\sup_{\ell\in\Z}
	 2^{\ell(\sigma_1+2)}
	 \norm{\dot\Delta_\ell u}_{L^1_tL^2_x}.
\end{equation*}
 
If $|\ell-k|\le C_0$, then
\begin{align}
&2^\ell
\norm{\dot\Delta_k\tau}_{L^\infty_tL^2_x}
\norm{\dot\Delta_\ell u}_{L^1_tL^2_x}
\notag\\
& \le
C2^k\bigl(2^{-k\scal}a_k\bigr)
\bigl(2^{-k(\sigma_1+2)}B_{\sigma_1}\bigr)
=C2^{-k(\sigma_1+\scal+1)}a_kB_{\sigma_1}
=C2^{-k(\sigma_1+d/2)}a_kB_{\sigma_1}.
\label{eq:Q-weight-C}
\end{align}
Since $|k|\le N$ and $\sigma_1+d/2>0$,
\begin{equation}
2^{-k(\sigma_1+d/2)}\le2^{N(\sigma_1+d/2)}.
\label{eq:finite-k-Q-C}
\end{equation}
Substituting  the estimates \eqref{eq:Q-weight-C}-
\eqref{eq:finite-k-Q-C} into \eqref{eq:Q-block-C} gives
\begin{equation}
{
\begin{aligned}
&2^{j\sigma_1}
\norm{\dot\Delta_j\cT Q(\tau^{(N)},\nabla u)}_{L^1_tL^2_x}
\\
&\le
C2^{(j+N+C_0)(\sigma_1+d/2)}
\norm{\tau}_{  X_c}
\norm{u}_{  Y_{\sigma_1}},
 j\le-N-J_* .
\end{aligned}}
\label{eq:Q-finite-C}
\end{equation}
For each fixed $N$, the right-hand side of \eqref{eq:Q-finite-C} tends to zero as $j\to-\infty$.
Consequently,
\begin{equation}
\cT Q(\tau^{(N)},\nabla u)^{\low}
\in\CL^1\bigl(\Besov{\sigma_1}{2}{c_0}\bigr).
\label{eq:Q-t-c0-C}
\end{equation}

For completeness, the commutator calculation and the transport continuity estimates are detailed below.  The remaining terms $\Pj(u\cdot\nabla u)$ and $u\cdot\nabla z$ are similar.

By the fixed-cutoff embeddings in the preceding estimate, taking $F=\tau$ in Lemma \ref{lem:zero-o-c} gives
\begin{align}
	\norm{[\cT,u\cdot\nabla]\tau}_{\CL^1(\Bdot^{\sigma_1}_{2,\infty})}
	\lesssim \norm{\tau}_{X_c}\norm{u}_{Y_{\sigma_1}}.
	\label{eq:renumbered-1}
\end{align}
This is the explicit realization of estimate \eqref{eq:bilinear-C} for the commutator term.
Truncate the critical $\ell^1$ factor by setting
	\begin{equation*}
		\tau^{(N)}:=\sum_{\abs{k}\le N}\dyad_k\tau.
	\end{equation*}

Write
\[
u_\ell:=\dyad_\ell u,
\qquad \tau_k:=\dyad_k\tau.
\]
If
$\abs{\ell-k}>C_0$, then the support of the corresponding product lies in an annulus of size
$2^{\max\{k,\ell\}}$. Hence one may choose a fixed constant $J_*>0$ such that, whenever
$j\le -N-J_*$,
\begin{equation*}
	\dyad_j[\cT,u\cdot\nabla]\tau^{(N)}
	=
	\sum_{\abs{k}\le N}
	\sum_{\abs{\ell-k}\le C_0}
	\dyad_j\Bigl(\cT(u_\ell\cdot\nabla\tau_k)
	-u_\ell\cdot\nabla \cT\tau_k\Bigr).
\end{equation*}
	The convolution kernel of $\dyad_j\cT$ has $L^2$ norm bounded by $C2^{jd/2}$, and $\cT\tau_k$ is supported at the same frequency scale as $\tau_k$, then
	\begin{align}
		&\norm{\dyad_j\cT(u_\ell\cdot\nabla\tau_k)}_2
		+\norm{\dyad_j(u_\ell\cdot\nabla \cT\tau_k)}_2
		\notag\\
		&\lesssim
		2^{jd/2}\norm{u_\ell\cdot\nabla\tau_k}_{L_x^1}
		+2^{jd/2}\norm{u_\ell\cdot\nabla \cT\tau_k}_{L_x^1}
		\notag\\
		&\lesssim
		2^{jd/2}2^k\norm{u_\ell}_2\norm{\tau_k}_2.
		\label{eq:renumbered-4}
	\end{align}
Taking the $L_t^1$ norm, using $L_t^\infty\times L_t^1$, and multiplying by $2^{j\sigma_1}$ to \eqref{eq:renumbered-4}, we obtain
	\begin{align*}
		&2^{j\sigma_1}
		\norm{\dyad_j[\cT,u\cdot\nabla]\tau^{(N)}}_{L_t^1L_x^2}
		\notag\\
		&\lesssim
		2^{j(\sigma_1+d/2)}
		\sum_{\abs{k}\le N}
		\sum_{\abs{\ell-k}\le C_0}
		2^k
		\norm{\tau_k}_{L_t^\infty L_x^2}
		\norm{u_\ell}_{L_t^1L_x^2}.
	\end{align*}
Since $s_c=d/2-1$ and $\abs{\ell-k}\le C_0$, we have
	\begin{equation*}
		2^k\norm{\tau_k}_{L_t^\infty L_x^2}
		\norm{u_\ell}_{L_t^1L_x^2}
		\lesssim
		2^{-k(\sigma_1+d/2)}a_kB_{\sigma_1}.
	\end{equation*}
    
For $\abs{k}\le N$,
\[
2^{-k(\sigma_1+d/2)}\le 2^{N(\sigma_1+d/2)}.
\]
Combining this with
$\sum_k a_k\lesssim\norm{\tau}_{X_c}$ and
$B_{\sigma_1}\lesssim\norm{u}_{Y_{\sigma_1}}$ yields
	\begin{equation}
		2^{j\sigma_1}
		\norm{\dyad_j[\cT,u\cdot\nabla]\tau^{(N)}}_{L_t^1L_x^2}
		\lesssim
		C\,2^{(j+N+C_0)(\sigma_1+d/2)}
		\norm{\tau}_{X_c}\norm{u}_{Y_{\sigma_1}},
		 j\le -N-J_*.
		\label{eq:renumbered-7}
	\end{equation}

For any fixed $N$, the right-hand side of \eqref{eq:renumbered-7} tends to zero as $j\to-\infty$. The global boundedness of the dyadic sequence follows from \eqref{eq:renumbered-1}. Consequently,
	\begin{equation*}
		\bigl([\cT,u\cdot\nabla]\tau^{(N)}\bigr)^\ell
		\in \CL^1\bigl(\Bdot^{\sigma_1}_{2,c_0}\bigr).
	\end{equation*}

\medskip
\noindent\textbf{Step 2: Passing  to the full source.}
The hybrid critical stress norm has an $\ell^1$ dyadic summation, as a result
\begin{equation*}
\norm{\tau-\tau^{(N)}}_{  X_c}\longrightarrow0
(N\to\infty).
\end{equation*}
The product estimate   gives
\begin{align*}
&\norm{\cT Q(\tau-\tau^{(N)},\nabla u)^{\low}}
_{\CL^1(\Besov{\sigma_1}{2}{\infty})}
 \le
C\norm{\tau-\tau^{(N)}}_{  X_c}
\norm{u}_{  Y_{\sigma_1}}
\longrightarrow0.
\end{align*}
Equivalently, the dyadic sequences
\begin{equation}
x_j^{(N)}:=2^{j\sigma_1}
\norm{\dot\Delta_j\cT Q(\tau^{(N)},\nabla u)}_{L^1_tL^2_x}
\label{eq:xjN-C}
\end{equation}
converge in $\ell^\infty(j\le j_0)$ to the analogous sequence $x_j$ for
$\cT Q(\tau,\nabla u)$. 
 By \eqref{eq:Q-t-c0-C}, each
$x^{(N)}$ belongs to $c_0$ tail.  Since $c_0$ is a closed subspace of $\ell^\infty$, we conclude that
\begin{equation*}
\cT Q(\tau,\nabla u)^{\low}
\in\CL^1\bigl(\Besov{\sigma_1}{2}{c_0}\bigr).
\end{equation*}

The passage from the truncated commutator to the full one follows from
\[
\|[\cT,u\cdot\nabla](\tau-\tau^{(N)})\|_{
\widetilde L^1(\dot B^{\sigma_1}_{2,\infty})}
\le C\|\tau-\tau^{(N)}\|_{X_c}\|u\|_{Y_{\sigma_1}}\longrightarrow0.
\]
Thus the full commutator has a $c_0$ tail by closedness of $c_0$
in $\ell^\infty$.  The other components in \eqref{eq:source-C} are handled by the same argument.  
\end{proof}

\subsection{An automatic Wiegner-type lemma}
\begin{lemma}
\label{lem:auto-wiegner}
Let
\begin{equation*}
-\frac d2<\sigma_1<\sigma<\scal,
\qquad
U_0^{\low}\in\Besov{\sigma_1}{2}{\infty},
\end{equation*}
and recall the definition of $R(t)$ in \eqref{eq:ULR}.  Then
\begin{equation}
{
\norm{R(t)}_{\cH^\sigma}
=o\!\left(t^{-(\sigma-\sigma_1)/2}\right)
(t\to\infty).}
\label{eq:D-result}
\end{equation}
\end{lemma}

\begin{proof}
By Lemma \ref{lem:c0-source}, there exists a bounded sequence $b_j\to0$ such
that
\begin{equation*}
2^{j\sigma_1}
\norm{\dot\Delta_j\cN(u,\tau)}_{L^1(0,\infty;L^2)}
=b_j,
\qquad j\le j_0.
\label{eq:bj-D}
\end{equation*}
The low-frequency estimate for the forced active semigroup and
\eqref{eq:DuhamelR} imply
\begin{equation}
a_j(t):=2^{j\sigma_1}\norm{\dot\Delta_jR(t)}_2
\le Cb_j
\quad \text{for all }t\ge0,
\label{eq:aj-D}
\end{equation}
and consequently
\begin{equation*}
\sup_{t\ge0}a_j(t)\longrightarrow0
(j\to-\infty).
\end{equation*}
Employing Theorem \ref{thm:nonlinear-upper} and Theorem \ref{thm:li-upp-} at the
critical active level gives
\begin{equation}
\norm{R^{\low}(t)}_{\Besov{\scal}{2}{1}}
\le C\langle t\rangle^{-(\scal-\sigma_1)/2}.
\label{eq:critical-R-D}
\end{equation}
Set
\begin{equation*}
\gamma:=\scal-\sigma>0,
\qquad
\widetilde s:=\sigma-\sigma_1,
\qquad
\alpha_s:=\frac{\tilde{s}}{2}=\alpha.
\end{equation*}

For $t\ge1$, choose $J=J(t)\in\Z$ such that
\begin{equation*}
2^J\le t^{-1/2}<2^{J+1}.
\end{equation*}
Fix $M\ge1$ and split the low-frequency Besov sum into three regions.

For $j\le J-M$, \eqref{eq:aj-D} gives
\begin{align}
\sum_{j\le J-M}2^{j\sigma}\norm{\dot\Delta_jR(t)}_2
&=\sum_{j\le J-M}2^{j\tilde{s}}a_j(t)
\notag\\
&\le C2^{(J-M)\tilde{s}}\sup_{j\le j_0}b_j
\le C2^{-M\tilde{s}}t^{-\alpha_s}.
\label{eq:low-zone-D}
\end{align}

For the moving block $|j-J|\le M$,
\begin{align*}
\sum_{|j-J|\le M}2^{j\sigma}\norm{\dot\Delta_jR(t)}_2
&\le C_M2^{J\tilde{s}}
\max_{|j-J|\le M}b_j
\notag\\
&=o(t^{-\alpha_s}),
\end{align*}
due to $J(t)\to-\infty$ and $b_j\to0$.
Finally, for the case $J+M\le j\le j_0$, using
\eqref{eq:critical-R-D} yields
\begin{align*}
\sum_{j\ge J+M}^{j_0}2^{j\sigma}\norm{\dot\Delta_jR(t)}_2
&\le 2^{-(J+M)\gamma}
\norm{R^{\low}(t)}_{\Besov{\scal}{2}{1}}
\notag\\
&\le C2^{-M\gamma}t^{-\alpha_s}.
\end{align*}
Therefore
\begin{equation}
\limsup_{t\to\infty}
t^{\alpha_s}\norm{R^{\low}(t)}_{\Besov{\sigma}{2}{1}}
\le C\bigl(2^{-M\tilde{s}}+2^{-M\gamma}\bigr).
\label{eq:limsup-D}
\end{equation}
Letting $M\to\infty$ implies
\begin{equation}
\norm{R^{\low}(t)}_{\Besov{\sigma}{2}{1}}
=o(t^{-\alpha_s}).
\label{eq:low-o-D}
\end{equation}
For the high-frequency part, the nonlinear critical decay theorem and the
exponential linear high-frequency estimate imply
(writing $R=(R_u,R_z)$ and denoting the corresponding high-frequency
energy sums by $\mathcal E_h$ and $\mathcal E_{L,h}$)
\begin{align}
\|R_u^h(t)\|_{\dot B^{\scal}_{2,1}}+\|R_z^h(t)\|_{\dot B^{\scal+1}_{2,1}}
&\le \mathcal E_h(t)+\mathcal E_{L,h}(t)
\notag\\
&\le C\langle t\rangle^{-(\scal-\sigma_1)/2}+Ce^{-ct}
=o(t^{-(\sigma-\sigma_1)/2}),
\label{eq:high-o-D}
\end{align}
where the last equality comes from  $\sigma<\scal$.  Combining
\eqref{eq:low-o-D} and \eqref{eq:high-o-D} gives
\eqref{eq:D-result}.
\end{proof}
\subsection{Proof of \cref{thm:non-trans}}
\begin{proof}[Proof of \cref{thm:non-trans}]
\textbf{Part \textup{(i)}.}
Assume \eqref{eq:negative-datum}.     \cref{thm:negative-p,thm:nonlinear-upper} give
\[
\norm{U(t)}_{\cH^\sigma}
\le C\langle t\rangle^{-(\sigma-\sigma_1)/2},
\]
which is \eqref{eq:upper-nonlinear}.
{Assume \eqref{eq:upper-nonlinear}, using \Cref{lem:same-order} 
yields $U_0^{\low}\in\Besov{\sigma_1}{2}{\infty}$. This is exactly
\eqref{eq:negative-datum}.
}

\textbf{Part \textup{(ii)}.}
Once part~\textup{(i)} has yielded
$U_0^{\low}\in\Besov{\sigma_1}{2}{\infty}$, applying
Lemma~\ref{lem:auto-wiegner} (which requires $\sigma<\scal$) gives
\[
\norm{U(t)-U_L(t)}_{\cH^\sigma}
=o\!\left(t^{-(\sigma-\sigma_1)/2}\right),
\]
which is \eqref{eq:auto-wiegner}.

\textbf{Part \textup{(iii)}}
Assume \eqref{eq:nondeg-datum}. By Definition \ref{def:frakB},
$U_0^{\low}\in\Besov{\sigma_1}{2}{\infty}$ implies
$U_0^{\low}\in\Besov{\sigma}{2}{1}$. \Cref{thm:linear-t-s} gives, for
sufficiently large $t$,
\begin{equation}
\norm{U_L^{\low}(t)}_{\Besov{\sigma}{2}{1}}
\ge c_\sharp t^{-\alpha}.
\label{eq:linear-lower}
\end{equation}
By part~\textup{(ii)},
\[
\norm{R^{\low}(t)}_{\Besov{\sigma}{2}{1}}
=o(t^{-\alpha}).
\]
Therefore
\[
\norm{U(t)}_{\cH^\sigma}
\ge \norm{U^{\low}(t)}_{\Besov{\sigma}{2}{1}}
\ge \bigl(c_\sharp-o(1)\bigr)t^{-\alpha}
\ge \frac{c_\sharp}{2}t^{-\alpha}
\]
for all sufficiently large $t$.  The upper bound follows from part~\textup{(i)}.
This proves \eqref{eq:two-sided-nonlinear}.

Assume \eqref{eq:two-sided-nonlinear}. Apply the eventual-time assertion of
\cref{lem:same-order} to the upper half of \eqref{eq:two-sided-nonlinear}.
This gives, without using part~\textup{(i)},
\begin{equation}
U_0^{\low}\in\Besov{\sigma_1}{2}{\infty}.
\label{eq:neg-first}
\end{equation}
then, it holds
$U_0^{\low}\in\Besov{\sigma}{2}{1}$.

\Cref{thm:nonlinear-upper} at the critical active level gives
\begin{equation*}
\|u^h(t)\|_{\dot B^{\scal}_{2,1}}+\|z^h(t)\|_{\dot B^{\scal+1}_{2,1}}
\le C t^{-(\scal-\sigma_1)/2}
=o(t^{-\alpha}),
\end{equation*}
  Hence the low-frequency part itself has the same lower order:
\begin{equation*}
\norm{U^{\low}(t)}_{\Besov{\sigma}{2}{1}}
\ge c t^{-\alpha}-o(t^{-\alpha})
\ge \frac c2t^{-\alpha}
\end{equation*}
for sufficiently large $t$.  Part~\textup{(ii)} using Lemma \ref{lem:auto-wiegner} yields
\begin{equation*}
\norm{U_L^{\low}(t)}_{\Besov{\sigma}{2}{1}}
\ge \norm{U^{\low}(t)}_{\Besov{\sigma}{2}{1}}
-\norm{R^{\low}(t)}_{\Besov{\sigma}{2}{1}}
\ge \frac c4t^{-\alpha}
\end{equation*}
for all sufficiently large $t$.   The necessity  of
Theorem \ref{thm:linear-t-s} therefore implies
$U_0^{\low}\in\frakB^{\sigma_1}_{2,\infty}$.  This completes the proof.
\end{proof}

\begin{remark}\label{rem:critical-level}
At $\sigma=\scal$, \cref{thm:non-trans} \textup{(i)} still gives the
upper decay equivalence
\begin{equation}
U_0^{\low}\in\Besov{\sigma_1}{2}{\infty}
\quad\Longleftrightarrow\quad
\norm{U(t)}_{\cH^{\scal}}
\lesssim \langle t\rangle^{-(\scal-\sigma_1)/2}.
\label{eq:critical-upp}
\end{equation}
Because the proof of Lemma~\ref{lem:auto-wiegner} relies on the strictly positive gap between $\scal$ and $\sigma$, it does not, by itself, establish
\[
\norm{U(t)-U_L(t)}_{\cH^{\scal}}
=o\!\left(t^{-(\scal-\sigma_1)/2}\right).
\]
Consequently, a two-sided characterization at the critical target level still requires a separate high-frequency or remainder separation estimate.
\end{remark}

\appendix
\section{{Paradifferential preliminaries}}\label{app:para}
In this section, we recall some definitions, notions, and properties of Besov spaces, which can be used repeatedly throughout this paper.
\subsection{Homogeneous Littlewood-Paley decomposition}

Choose radial functions $\chi,\varphi\in C_c^\infty(\R^d)$ such that
\[
\supp\chi\subset\{ |\xi|\le4/3\},\qquad
\supp\varphi\subset\{3/4\le|\xi|\le8/3\},
\]
and
\[
\chi(\xi)+\sum_{j\ge0}\varphi(2^{-j}\xi)=1,
\qquad
\sum_{j\in\Z}\varphi(2^{-j}\xi)=1\quad(\xi\ne0).
\]
Set
\[
\dyad_jf=\cF^{-1}\bigl(\varphi(2^{-j}\cdot)\widehat f\bigr),
\qquad
\lowcut_jf=\sum_{k<j}\dyad_kf.
\]
We work in the standard homogeneous distribution space $\cS'_h$; see \cite[Chapter~2]{BCD2011}.

For $s\in\R$ and $1\le r\le\infty$,
\begin{equation*} 
\norm f_{\Bdot^s_{2,r}}
=\norm{\bigl(2^{js}\norm{\dyad_jf}_{L^2}\bigr)_{j\in\Z}}_{\ell^r}.
\end{equation*}

For $T>0$ and $1\le q\le\infty$, the Chemin-Lerner norm is
\begin{equation*}
\norm f_{\CL^q_T(\Bdot^s_{2,r})}
=\norm{\bigl(2^{js}\norm{\dyad_jf}_{L^q(0,T;L^2)}\bigr)_{j\in\Z}}_{\ell^r}.
\end{equation*}
Fix an integer $j_0$ once and for all. It will be chosen sufficiently negative in the low-frequency semigroup analysis. Define
\begin{equation*} 
f^\low=\sum_{j\le j_0}\dyad_jf,
\qquad f^\high=f-f^\low.
\end{equation*}
We use the corresponding seminorms
\[
\norm f_{\Bdot^{s,\low}_{2,r}}
=\norm{\bigl(2^{js}\norm{\dyad_jf}_{L^2}\bigr)_{j\le j_0}}_{\ell^r},
\qquad
\norm f_{\Bdot^{s,\high}_{2,r}}
=\norm{\bigl(2^{js}\norm{\dyad_jf}_{L^2}\bigr)_{j\ge j_0-1}}_{\ell^r}.
\]
Define the two-index hybrid Besov norm
\begin{equation}\label{eq:h-two-in}
\norm f_{\Bdot^{s_\ell,s_h}_{2,1}}
:=\sum_{j\le j_0}2^{js_\ell}\norm{\dyad_jf}_2
+\sum_{j\ge j_0-1}2^{js_h}\norm{\dyad_jf}_2,
\end{equation}
where the finitely many overlapping shells may be counted in both sums; this only changes the norm by an equivalent constant.

\subsection{Bernstein, product, and commutator estimates}
We use Bony's decomposition
\[
\begin{aligned}
	fg &= \dot T_f g + \dot T_g f + \dot R(f,g),\\
	\dot T_f g &= \sum_k \dot S_{k-1} f \, \dot\Delta_k g,\\
	\dot R(f,g) &= \sum_k \dot\Delta_k f \,
	(\dot\Delta_{k-1}+\dot\Delta_k+\dot\Delta_{k+1})g.
\end{aligned}
\]
All estimates are first established for finite dyadic sums and then
extended by convergence in the indicated homogeneous spaces.
\begin{lemma}[\cite{BCD2011}]\label{lem:m-product}
Let $-d/2<s<d/2$ and $r\in\{1,\infty\}$. Then
\begin{equation*} 
\norm{fg}_{\Bdot^s_{2,r}}
\le C\norm f_{\Bdot^{d/2}_{2,1}}\norm g_{\Bdot^s_{2,r}}.
\end{equation*}
The same estimate holds with all norms restricted to low-frequency on the left and with the full norms on the right.
\end{lemma}
 
\begin{lemma}[\cite{BCD2011}]\label{lem:bal-p-1}
Let $-d/2<s<d/2-1$. Then
\begin{equation*} 
\norm{fg}_{\Bdot^s_{2,\infty}}
\le C\norm f_{\Bdot^{d/2-1}_{2,1}}\norm g_{\Bdot^{s+1}_{2,\infty}}.
\end{equation*}

In particular, for the bilinear map $Q$  in \eqref{eq:parametric-GOB}
and any smooth homogeneous order-zero multiplier $M(D)$,
\begin{equation}\label{eq:Q-product}
\norm{M(D)Q(F,\nabla v)}_{\Bdot^s_{2,\infty}}
\le C\norm F_{\Bdot^{d/2-1}_{2,1}}\norm v_{\Bdot^{s+2}_{2,\infty}}.
\end{equation}
\end{lemma}
\begin{lemma}[\cite{BCD2011}]\label{lem:tran-c-1}
Let $-d/2<s<d/2+1$, $v$ be divergence-free, and $f\in\Bdot^s_{2,\infty}$. There exists a bounded nonnegative sequence $(c_j)_{j\in\Z}$ with $\sup_jc_j\le1$ such that
\begin{equation*}
\norm{[\dyad_j,v\cdot\nabla]f}_{L^2}
\le Cc_j2^{-js}\norm{\nabla v}_{\Bdot^{d/2}_{2,1}}\norm f_{\Bdot^s_{2,\infty}}.
\end{equation*}
If $f\in\dot B^s_{2,1}$, the estimate holds with
$\|f\|_{\dot B^s_{2,1}}$ on the right and
$(c_j)\in\ell^1$, $\|c\|_{\ell^1}\le1$.
\end{lemma}

\begin{lemma}[cf.\ \cite{BCD2011}]\label{lem:zero-o-c}
Let $T(D)$ be a smooth homogeneous Fourier multiplier of order zero. If $-d/2<s<d/2-1$, then
\begin{equation}\label{665453}
\norm{[T(D),v\cdot\nabla]F}_{\Bdot^s_{2,\infty}}
\le C\norm F_{\Bdot^{d/2-1}_{2,1}}
\norm v_{\Bdot^{s+2}_{2,\infty}}.
\end{equation}
\end{lemma}
\begin{proof}[Sketch of proof]
Decompose $v\cdot\nabla F=\dot T_v\nabla F+\dot T_{\nabla F}v+\dot R(v,\nabla F)$. The term $[T(D),\dot T_v]\nabla F$ is a paraproduct commutator with an order-zero multiplier: writing $\dyad_k[T(D),\dot S_{k-1}v]\dyad_k\nabla F$ as a convolution with the kernel $K_k$ of $T(D)\widetilde{\dyad}_k$ and using $\int|y||K_k(y)|\,\dd y\lesssim2^{-k}$, one obtains
\[
\norm{[T(D),\dot T_v]\nabla F}_{\Bdot^s_{2,\infty}}
\lesssim\sup_k 2^{-k}\norm{\nabla\dot S_{k-1}v}_{L^\infty}\,2^{ks}\norm{\dyad_k\nabla F}_2
\lesssim\norm v_{\Bdot^{s+2}_{2,\infty}}\norm F_{\Bdot^{d/2-1}_{2,1}},
\]
where $\norm{\nabla\dot S_{k-1}v}_{L^\infty}\lesssim2^{k(d/2-1-s)}\norm v_{\Bdot^{s+2}_{2,\infty}}$ requires $s<d/2-1$. The remaining four terms, $T(D)\dot T_{\nabla F}v$, 
$\dot T_{\nabla(T(D)F)}v$, $T(D)\dot R(v,\nabla F)$ and $\dot R(v,\nabla T(D)F)$, carry no commutator structure and are bounded by the standard paraproduct and remainder estimates \cite[Section~2.6]{BCD2011}, using the embedding $\Bdot^{d/2-2}_{2,1}\hookrightarrow\Bdot^{-2}_{\infty,1}$ for $\nabla F$ and the condition $s+d/2>0$ for the remainder.
\end{proof}

\section{Critical hybrid product estimates}\label{app:critical-hybrid-products}
This appendix {gives} the hybrid product and localized commutator
estimates used in \eqref{63634}. The
high-frequency stress index is one derivative above the velocity
index, so the placement of each projection and derivative is made explicit.
\begin{lemma}[\cite{ChenHao2019}]\label{lem:critical-h-m}

Let $d\ge2$.  Then
\begin{equation*} 
\norm{fg}_{\Bdot^{d/2-1,d/2}_{2,1}}
\le C\norm f_{\Bdot^{d/2}_{2,1}}
\norm g_{\Bdot^{d/2-1,d/2}_{2,1}}.
\end{equation*}
The same estimate holds componentwise for vector and tensor products.
\end{lemma}

\begin{corollary}[\cite{ChenHao2019}]\label{cor:critical-h-O}

Let $Q$ be the bilinear term in \eqref{eq:parametric-GOB}. Then
\begin{align}
\norm{Q(\tau,\nabla u)}_{\Bdot^{d/2-1,d/2}_{2,1}}
&\le C\norm\tau_{\Bdot^{d/2-1,d/2}_{2,1}}
\norm u_{\Bdot^{d/2+1}_{2,1}},
\label{eq:hyb-Q-app}\\
\norm{(\nabla u)\tau}_{\Bdot^{d/2-1,d/2}_{2,1}}
&\le C\norm\tau_{\Bdot^{d/2-1,d/2}_{2,1}}
\norm u_{\Bdot^{d/2+1}_{2,1}}.
\label{eq:hyb-g-app}
\end{align}
Consequently,
\begin{align}
&\norm{(\Lambda^{-1}\mathbb P\diver Q(\tau,\nabla u))^\low}
_{\Bdot^{d/2-1}_{2,1}}
+\norm{(\mathbb P\diver Q(\tau,\nabla u))^\high}
_{\Bdot^{d/2-1}_{2,1}}
\notag\\
&\qquad\le C\norm\tau_{\Bdot^{d/2-1,d/2}_{2,1}}
\norm u_{\Bdot^{d/2+1}_{2,1}}.
\label{eq:hyb-Q-p-app}
\end{align}
For each of the vector fields
$G_1=\nabla u\cdot\nabla\tau$ and
$G_2=\nabla u\cdot\nabla\Delta^{-1}\operatorname{div}\operatorname{div}\tau$,
the corresponding estimate is
\begin{equation}\label{eq:hyb-p-app}
\|(\Lambda^{-1}\mathbb P G_i)^\ell\|_{\dot B^{\scal}_{2,1}}
+\|(\mathbb P G_i)^h\|_{\dot B^{\scal}_{2,1}}
\le C\|\tau\|_{\dot B^{\scal,\scal+1}_{2,1}}
\|u\|_{\dot B^{\scal+2}_{2,1}},\qquad i=1,2.
\end{equation}
Here $\operatorname{div}u=0$ is required for the divergence identities
used to estimate $G_i$.
\end{corollary}

\begin{lemma}[\cite{BCD2011,Zhai2021}]\label{lem:actives-com}

Let $u$ be divergence-free.  Then
\begin{align*}
&\sum_{j\in\mathbb Z}2^{j(d/2-1)}
\norm{[\dyad_j,u\cdot\nabla]f}_2
\le C\norm{\nabla u}_{\Bdot^{d/2}_{2,1}}
\norm f_{\Bdot^{d/2-1}_{2,1}},
\\
&\norm{([\Lambda^{-1}\mathbb P\diver,u\cdot\nabla]\tau)^\low}
_{\Bdot^{d/2-1}_{2,1}}
\le C\norm u_{\Bdot^{d/2+1}_{2,1}}
\norm\tau_{\Bdot^{d/2-1,d/2}_{2,1}},\\
&\sum_{j\le j_0}2^{j(d/2-1)}
\norm{[\Lambda,u\cdot\nabla]\dyad_j u}_2
\le C\norm{\nabla u}_{\Bdot^{d/2}_{2,1}}
\norm{u^\low}_{\Bdot^{d/2-1}_{2,1}}.
\end{align*}
\end{lemma}
\begin{proof}[Sketch of proof]
The first estimate is the classical critical commutator estimate \cite[Lemma~2.100]{BCD2011}. The second is the hybrid, $\ell^1$ version of Lemma~\ref{lem:zero-o-c} at the endpoint $s=d/2-1$: the proof is identical, the only difference being that at this endpoint the bound $\norm{\nabla\dot S_{k-1}\tau}_{L^\infty}\lesssim 2^{2k}\norm{\tau}_{\Bdot^{d/2-1}_{2,1}}$ is used together with the $\ell^1$ summability of $2^{k(d/2+1)}\norm{\dyad_ku}_2$.
For the third estimate, write $u\cdot\nabla u_j=\dot T_u\nabla u_j+\dot T_{\nabla u_j}u+\dot R(u,\nabla u_j)$ with $u_j=\dyad_ju$. The paraproduct commutator $[\Lambda,\dot T_u]\nabla u_j$ involves a multiplier of order one and satisfies $\norm{[\Lambda,\dot T_u]\nabla u_j}_2\lesssim\norm{\nabla u}_{L^\infty}\norm{\nabla u_j}_2\lesssim2^j\norm{\nabla u}_{\Bdot^{d/2}_{2,1}}\norm{u_j}_2$ \cite[Lemma~2.99]{BCD2011}, which gives the claim for this term after multiplication by $2^{j(d/2-1)}$ and summation. For the remaining terms, $\Lambda$ costs a factor $2^k$ on the blocks $\dyad_ku$, $k\ge j-3$, while $\norm{\dot S_{k-1}\nabla u_j}_{L^\infty}+\norm{\nabla u_j}_{L^\infty}+2^{-j}\norm{\nabla\Lambda u_j}_{L^\infty}\lesssim2^{j(1+d/2)}\norm{u_j}_2$; hence
\[
\begin{aligned}
	&2^{j(d/2-1)}
	\Bigl(
	\norm{\Lambda\dot T_{\nabla u_j}u}_2
	+\norm{\dot T_{\nabla\Lambda u_j}u}_2\\
	&\qquad
	+\norm{\Lambda\dot R(u,\nabla u_j)}_2
	+\norm{\dot R(u,\nabla\Lambda u_j)}_2
	\Bigr)\\
	&\qquad\lesssim
	2^{jd/2}\norm{u_j}_2
	\sum_{k\ge j-3}
	2^{(j-k)d/2}\,
	2^{k(d/2+1)}\norm{\dyad_k u}_2.
\end{aligned}
\]
Summing over $j\le j_0$ by Young's inequality for sequences and using $2^{jd/2}\le2^{j(d/2-1)}$ for $j\le0$ yields the bound $C\norm{u^\low}_{\Bdot^{d/2-1}_{2,1}}\norm{u}_{\Bdot^{d/2+1}_{2,1}}$, which is the claim.
\end{proof}

\begin{lemma} \label{key-le-10}
	Let $\gamma_1>0$ and $\gamma_2>0$. Suppose that
	nonnegative sequences $(x_j)_{j\le j_0}$ and $(e_j)_{j\le j_0}$ satisfy
	\begin{equation}\label{eq:B-v-hyp}
		 	\sum_{j\le j_0}2^{j\gamma_1}x_j<\infty,
		 \qquad e\in\ell^1,
	\end{equation}
	and
	\begin{equation}	\label{eq:B-v-ineq}
		x_j\le A+C_0e_j\sum_{k<j}2^{-(j-k)\gamma_1}x_k
	+C_0\sum_{k\ge j}2^{-(k-j)\gamma_2}e_kx_k.
	\end{equation}
	Then there exists $e_*>0$, depending only on $C_0$ and $\gamma_1$, such that if
	$\norm{e}_{\ell^1}\le e_*$, then
	$$\sup_{j\le j_0}x_j\le C A$$
	for some $C>0$ depending only on $C_0$, $\gamma_1$ and $\norm{e}_{\ell^1}$.
\end{lemma}

\begin{proof}
	Set $q:=2^{-\gamma_1}\in(0,1)$ and
	\[
	B_j:=\sup_{j\le m\le j_0}x_m,
	\qquad H_j:=\sum_{k<j}q^{j-k}x_k,
	\qquad \delta_j:=\sum_{m\ge j}q^{m-j}e_m.
	\]
	For $j\le m\le j_0$, split the first sum in \eqref{eq:B-v-ineq} at $j$:
	\[
	\sum_{k<m}q^{m-k}x_k
	=q^{m-j}H_j+\sum_{j\le k<m}q^{m-k}x_k
	\le q^{m-j}H_j+\frac{q}{1-q}B_j,
	\]
	and bound the second sum by $B_j\sum_{k\ge m}e_k\le\norm{e}_{\ell^1}B_j$. Since $e_mq^{m-j}\le\delta_j$ for $m\ge j$, taking the supremum over
	$j\le m\le j_0$ gives
	\[
	B_j\le A+C\norm{e}_{\ell^1}B_j+C\delta_jH_j.
	\]
	After decreasing $e_*$, the middle term is absorbed, and hence
	$B_j\le CA+C\delta_jH_j$. 
    
    Since
	$H_{j+1}=qH_j+q x_j$, it follows that
	\[
	H_{j+1}\le q(1+C\delta_j)H_j+CqA.
	\]
	Iteration from $J<j$, 	we obtain
	\begin{align}
		H_j
		&\le
		q^{j-J}
		\prod_{\ell=J}^{j-1}\bigl(1+C\delta_\ell\bigr)H_J
		+CqA\sum_{m=J}^{j-1}
		q^{j-1-m}
		\prod_{\ell=m+1}^{j-1}
		\bigl(1+C\delta_\ell\bigr).
		\label{eq:Hj-iteration}
	\end{align}
	By the definition
	\[
	\delta_\ell
	:=\sum_{m\ge \ell}q^{m-\ell}e_m,
	\]
	Tonelli's theorem gives
	\begin{align*}
		\sum_{\ell\le j_0}\delta_\ell
		&=
		\sum_{\ell\le j_0}
		\sum_{\substack{m\ge\ell\\m\le j_0}}
		q^{m-\ell}e_m
		\nonumber\\
		&=
		\sum_{m\le j_0}e_m
		\sum_{\ell\le m}q^{m-\ell}
		=
		\frac{1}{1-q}
		\sum_{m\le j_0}e_m
		<\infty.
	\end{align*}
	Since \(\log(1+s)\le s\) for \(s\ge0\), it follows that,
	for all \(a\le b\le j_0\),
	\[
	\prod_{\ell=a}^{b}\bigl(1+C\delta_\ell\bigr)
	\le
	\exp\left(
	C\sum_{\ell=a}^{b}\delta_\ell
	\right)
	\le
	\exp\left(
	\frac{C}{1-q}
	\sum_{m\le j_0}e_m
	\right)
	=:C_*.
	\]
	In particular, \(C_*\) is independent of \(a\), \(b\), \(J\),
	and \(j\). Therefore, \eqref{eq:Hj-iteration} implies
	\begin{align*}
		H_j
		&\le
		C_*q^{j-J}H_J
		+
		CC_*qA\sum_{m=J}^{j-1}q^{j-1-m}
		\nonumber\\
		&\le
		C_*q^{j-J}H_J
		+
		\frac{CC_*q}{1-q}A
		\le
		Cq^{j-J}H_J+CA.
	\end{align*}
    
Therefore it yields
	\[
	H_j\le Cq^{j-J}H_J+CA.
	\]
	For fixed $j$,
	\[
	q^{j-J}H_J
	=2^{-j\gamma_1}
	\sum_{k<J}2^{k\gamma_1}x_k
	\longrightarrow0
	\qquad(J\to-\infty),
	\]
	by \eqref{eq:B-v-hyp}. Thus $H_j\le CA$ and then $B_j\le CA$.
\end{proof}

	\section*{Acknowledgements.}
	Z. Chen is supported by National Natural Science Foundation of China (No. 12501294).
	M. Fei is supported partly by National Natural Science Foundation of China (No.12271004, No.12471222) and Natural Science Foundation of Anhui 
Province of China (No.2308085J10). L. Liu is supported by National Natural Science Foundation of China (No.12571243) and Natural Science Foundation of Anhui Province (No.2508085Y004). J. Wu is supported by the National Science Foundation of the United States (DMS 2104682 and DMS 2309748).{ The authors acknowledge the assistance of ChatGPT (OpenAI) with
language polishing and the exposition of selected arguments.
The authors take full responsibility for the mathematical content
and the final version of the manuscript.}
	
	\section*{Declaration }
	\textbf{Conflict of interest:} The authors declared that they have no conflicts of interest to this work.

 

\end{document}